\documentclass[12pt]{amsart}

\usepackage[usenames,dvipsnames,svgnames,table]{xcolor}
\usepackage{graphicx}
\usepackage{inputenc}
\usepackage{amsmath, latexsym, amsfonts, amssymb, amsthm, amscd}
\usepackage{hyperref}

\usepackage{caption}
\usepackage{subcaption}
\usepackage[left=3cm, right=2.5cm, top=2.5cm, bottom=2.5cm]{geometry}

\newtheorem{proposition}{Proposition}
\newtheorem{lem}{Lemma}

\newtheorem{teo}{Theorem}

\newcommand{\p}{\mathbb{P}}

\newcommand{\e}{\mathbb{E}}

\def\indi{\mathbf{1}}

\def\r{\mathbb{R}}
\def\n{\mathbb{N}}

\def\z{\mathbb{Z}}
\def\indi{\mathbf{1}}
\newcommand{\ud}[1]{\, \mathrm{d}#1}

\newcommand{\Exp}[1]{\mathbb{E}\left[#1\right]}
\newcommand{\Expcon}[2]{\mathbb{E}\left[#1\mid #2\right]}

\newcommand{\expo}[1]{\exp\left\{#1\right\}}

\begin{document}

\title[Yaglom's limit for Galton-Watson in varying environment]{ Yaglom's limit for critical Galton-Watson processes in varying environment: A probabilistic approach}
 
\author{Natalia Cardona-Tob\'on}
\address{Centro de Investigaci\'on en Matem\'aticas. Calle Jalisco s/n. C.P. 36240, Guanajuato, M\'exico}
\email{natalia.cardona@cimat.mx}
 
\author{Sandra Palau}
\address{IIMAS, Universidad Nacional Aut\'onoma de M\'exico. CDMX, 04510, Ciudad de M\'exico, M\'exico}
\email{sandra@sigma.iimas.unam.mx}
 
 
\maketitle

\begin{abstract}
A Galton-Watson process in varying environment is a discrete time branching process where the offspring distributions vary among generations. Based on a two-spine decomposition technique, we provide a probabilistic argument of a Yaglom-type limit for this family processes. The result states that, in the critical case, a suitable normalisation of  the process conditioned on non-extinction converges in distribution to a standard exponential random variable. Recently, this result has been established by Kersting   [{\it J. Appl. Probab.} {\bf57}(1), 196--220, 2020] using analytic techniques.
\bigskip

 
\medskip
 
\end{abstract}

\section{Introduction}
A Galton-Watson process in varying environment (GWVE) is a discrete time branching process where the offspring distributions vary among generations, in other words individuals give birth independently and their offspring distributions coincide within each generation. More precisely, a {\it varying environment} is a sequence \ $Q=(q_1,q_2,\ldots)$ \ of probability measures on \ $\n_0=\{0,1,2,\ldots\}$. \  A  {\it Galton-Watson process} \ $Z^Q = \{Z_n^Q: n\geq 0\}$ \ {\it in  a varying environment} \ $Q$ \ is a Markov chain defined recursively as follows 
$$ Z_0^Q=1\quad \mbox{ and } \quad Z_n^Q
=\underset{i=1}{\overset{ Z^Q_{n-1}}{\sum}}\chi_{i}^{(n)}, \qquad n\geq 1,$$
where \  $\{\chi_{i}^{(n)}: i,n\geq 1\}$ \ is a sequence of independent random variables
such that
$$\p(\chi_{i}^{(n)}=k)=q_n(k), \qquad \quad k\in \n_{0},\ i,\ n\geq 1.$$ 
The variable  \ $\chi_{i}^{(n)}$ \ denotes the offspring of the \ $i$-th individual in the \ $(n-1)$-th generation. Its generating function is given by 
$$ f_n(s):= \Exp{s^{\chi_{i}^{(n)}}}=\sum_{k=0}^{\infty} s^k q_n(k), 
\qquad 0 \leq s\leq 1, \ n\geq 1.$$ 

Hence, by applying the branching property recursively, we deduce that the generating function of \ $Z_n^Q$ \ is given in terms of \ $(f_1, f_2,\ldots)$ \ as follows 
\begin{equation}\label{laplace Z}
\Exp{s^{Z_n^Q}}=f_1\circ\cdots\circ f_n(s), 
\qquad 0 \leq s\leq 1, \ n\geq 1,
\end{equation}
where \ $f\circ g$ \ denotes the composition of \ $f$ \ with \ $g$.\

\noindent Moreover, by differentiating in \ $s$,  \ we obtain
\begin{equation}\label{eq_media_Zn}
\e[Z_n^Q]=\mu_n,\qquad  \mbox{and} \qquad
\frac{\e[Z_n^Q(Z_n^Q-1)]}{\e[Z_n^Q]^2}= \sum_{k=0}^{n-1}\frac{\nu_{k+1}}{\mu_{k}},
\qquad n\geq 1,
\end{equation}
where \ $\mu_0:=1$\  and for any \ $n\geq 1$,
\begin{equation}\label{defi mu}
\mu_n:= f_1'(1)\cdots f_n'(1), \qquad 
\mbox{ and }\qquad \nu_n:= \frac{f_n''(1)}{f_n'(1)^2}=\frac{\text{Var}\left[\chi_{i}^{(n)}\right]}{\Exp{\chi_{i}^{(n)}}^2}+\left(1-\frac{1}{\Exp{\chi_{i}^{(n)}}}\right),
\end{equation}
where \ Var$\left[\chi_{i}^{(n)}\right]$ \ is the variance of the variable. 
For further details about GWVEs, we refer to the monograph of Kersting and Vatutin \cite{kersting2017discrete}.

According with Kersting, \cite{kersting2017unifying}, we say that a  GWVE is {\it regular} if there exists a constant \ $c>0$ \ such that for all \ $n\geq 1$,
$$\Exp{(\chi_{i}^{(n)})^2 \indi_{\{\chi_{i}^{(n)}\geq 2\}}} \leq c \Exp{\chi_{i}^{(n)} \indi_{\{\chi_{i}^{(n)}\geq 2\}}} \Exp{\chi_{i}^{(n)}\left| \chi_{i}^{(n)}\geq 1\right.}.$$
He proved that a regular GWVE has extinction a.s (i.e. \ $\p( Z_n^Q=0 \mbox{ for some }n)=1$) \ if and only if 
\  $  \sum_{k=0}^{\infty}\frac{\nu_{k+1}}{\mu_{k}}=\infty$ \ or 
\  $\mu_n\rightarrow 0$\ as\ $n\to \infty$,  \cite[ Theorem 1]{kersting2017unifying}. \ 
In addition, he gave the following classification. \\

A regular GWVE is 
\begin{enumerate}
	\item[i.] {\it supercritical} if and only if	
	\ $\underset{k=0}{\overset{\infty}\sum} \frac{\nu_{k+1}}{\mu_{k}}<\infty$ \ and 
	\ $\underset{n\rightarrow\infty}{\lim}\mu_n =\infty$,  
	\item[ii.] {\it asymptotically  degenerate} if and only if \ $\underset{k=0}{\overset{\infty}\sum} \frac{\nu_{k+1}}{\mu_{k}}<\infty$ \ and 
	\ $0<\underset{n\rightarrow\infty}{\lim}\mu_n< \infty$, 
	\item[iii.] {\it critical} if and only if
	\ $\underset{k=0}{\overset{\infty}\sum} \frac{\nu_{k+1}}{\mu_{k}}=\infty$ \ and 
	\ $\underset{n\rightarrow\infty}{\lim}\mu_n\underset{k=0}{\overset{n-1}\sum} \frac{\nu_{k+1}}{\mu_{k}}=\infty$,   
	\item[iv.] {\it subcritical } if and only if 
	\ $\underset{n\rightarrow\infty}{\liminf\mu_n}=0$ \ and
	\ $\underset{n\rightarrow\infty}{\liminf}\mu_n\underset{k=0}{\overset{n-1}\sum} \frac{\nu_{k+1}}{\mu_{k}}<\infty$.
\end{enumerate}

Kersting's definition is an extension of the classical categorisation of branching processes. Indeed, when the environment is constant, we have \ $\mu_k=\mu^k$\  and \ $\nu_k=\sigma^2/\mu^2 + (1-1/\mu)$,\  for \ $k\geq 1$, \ where \ 
$\mu$\  and \ $\sigma^2$ \ are the mean and variance of the offspring distribution, respectively; we recover the original classification. We observe that in this case, the asymptotically degenerate case is not possible. 

Given a varying environment \ $Q$,\  we  define the sequence \ $\{a_n^Q: n\geq 0\}$ \ as follows 
\begin{equation*}
a_0^Q=1, \qquad \mbox{and}\qquad a_n^Q= \frac{\mu_n}{2}\sum_{k=0}^{n-1}\frac{\nu_{k+1}}{\mu_{k}},\qquad n\geq 1.
\end{equation*}
Kersting, \cite[Theorem 4]{kersting2017unifying}, showed that in the critical regime, \ $a_n^Q\rightarrow \infty$ \ and that
\begin{equation}\label{limit a}
\lim\limits_{n\to \infty}\frac{a_n^Q}{\mu_n} \p(Z_n^Q>0)=1. 
\end{equation}
This asymptotic behavior is  a generalization of Kolmogorov's theorem for Galton-Watson processes with constant environment (see \cite{kolmogorov1938losung}).

In the rest of the paper, we work with regular critical GWVE. Further, we assume the following condition 
\begin{equation}\label{eq_cond_mild}\tag{\bf{A}}
\mbox{there exists } c>0 \mbox{ such that } \ 
f_n'''(1)\leq c f_n''(1)(1+ f_n'(1)), \ \mbox{ for any } n\geq 1.
\end{equation} 

Kersting proved that this condition implies that the GWVE is regular,  see \cite[Proposition 2]{kersting2017unifying}. Moreover, he explained that Condition \eqref{eq_cond_mild} is a rather mild condition. Indeed, it is satisfied by most common probability distributions, for instance the Poisson, binomial, geometric, hypergeometric,  and  negative binomial distributions. Another important example satisfying Condition \eqref{eq_cond_mild} are random variables that are a.s. uniformly bounded by a constant.

We are ready to present our main result, which is in accordance with  Yaglom's theorem for classical Galton-Watson processes. 

\begin{teo}[Yaglom's limit]\label{teo_yaglom}
	Let \ $\{Z_n^Q: n\geq 0\}$ \ be a critical GWVE that satisfies Condition \eqref{eq_cond_mild}. Then
	\begin{equation*}
	\left(\frac{Z_n^Q}{a_n^Q}; \p(\ \cdot\ | Z_n^Q > 0)\right) \stackrel{(d)}{\longrightarrow}  \left(Y; \p\right) , \qquad \mbox{ as }\ n\rightarrow \infty,
	\end{equation*}
	where \ $Y$ \ is a standard exponential random variable.
\end{teo}

In the classical theory with constant environment, this result has several proofs, the first one was given  by Yaglom \cite{yaglom1947certain}. In \cite{ lyons1995conceptual},  a probabilistic proof via a characterisation of the exponential distribution was presented. Later on, Geiger characterised the exponential random variable by a distributional equation and  he presented another proof of Yaglom's limit based on that equation (see  \cite{geiger1999elementary,geiger2000new}).   Recently, Ren et al. \cite{ren20172}, developed yet another new proof using a two-spine decomposition technique.

When the environment is varying, Jagers \cite{jagers1974galton} proved the  convergence under extra assumptions. Afterwards, Bhattacharya and Perlman \cite{bhattacharya2017time} obtained the same result with  weaker assumptions than Jagers (but stronger than ours).  Kersting  \cite{kersting2017unifying}  provided yet another proof in a similar framework to ours, that we will explain below. An extension in the presence of immigration and the same setting as Kersting's has been established in \cite{gonzalez2019branching}.  A multi-type version with analogous  assumptions as  Kersting's can be found in  \cite{dolgopyat2018}.  All these authors established the exponential convergence using an analytical approach. The condition in Kersting \cite{kersting2017unifying} is the following.  For every \ $\epsilon >0$ \ there is a constant \ $c_\epsilon< \infty$ \ such that 
\begin{equation*}
\Exp{(\chi_{i}^{(n)})^2\indi_{\left\{\chi_{i}^{(n)} > c_\epsilon(1 + \e[\chi_{i}^{(n)}])\right\}}}\leq \epsilon \Exp{(\chi_{i}^{(n)})^2\indi_{\{  \chi_{i}^{(n)} \geq 2\}}},\qquad\qquad \mbox{for any } n\geq 1.
\end{equation*}
He explained that a direct verification of his assumptions can be cumbersome. Therefore, he introduced  Condition \eqref{eq_cond_mild} as an assumption easier to handle that implies the latter condition. For this reason, we  prefer to work directly under the Assumption \eqref{eq_cond_mild}, which is good enough for our purposes. 

In this manuscript, we give a probabilistic argument of  Yaglom's limit for GWVE. It is based on a two spine decomposition method and a characterisation of the exponential distribution via a size-biased transform and is close in spirit to that of \cite{ren20172}. A one-spine decomposition is already known in the literature (see \cite[Section 1.4]{kersting2017discrete}). We believe that it is possible to use a one-spine decomposition to prove Yaglom-type limit for GWVE, but we could not find a proof with this approach in the literature.  However, we decided to tackle the proof with a two spines decomposition. The reason comes from the classical theory of a Galton-Watson process in constant environment. Consider the most recent common ancestor (MRCA) of the particles at generation \ $n$. \ When the environment is constant, Geiger \cite{geiger2000new} showed that conditioned on the event of non-extinction at generation \  $n$,\   asymptotically there are exactly two children
of the MRCA with at least one descendant at generation \ $n$. Based on this intuition, it is natural to consider a two spine decomposition whose spines correspond to the genealogical lines of these two individuals.

The authors in \cite{ren20172} created a two-spine decomposition technique for Galton-Watson processes in constant environment that cannot be translated directly into our settings. Here, associated to each \ $Z_n^Q$, \ we construct a Galton-Watson tree in varying environment up to time \ $n$ \ with two marked genealogical lines. This tree can be decomposed in subtrees along these lines. A key point is the distribution  of the generation of the most recent common ancestor of these genealogical lines, \ $K_n$.  \ When the environment is constant, \ $K_n$ \ has uniform distribution in \ $\{0,\dots, n-1\}$ and the subtrees are independent Galton-Watson trees.  When the environment varies, this last property does not hold anymore.  In order to match the above decomposition  with that at the exponential distribution, it is fundamental to know the law of \ $K_n$ \ explicitly. Thus,  we determine the distribution of  \ $K_n$ \ that makes the method work. Moreover, we identify the subtrees with Galton-Watson trees in a modified environment. In the next section, we explain this in further detail.

Our contribution is that our proof provides further understanding on why the limit must be an exponential random variable.  An important part of our approach is in studying random trees and being able to adequately select inside them two marked genealogical lines. We believe that one can adapt this decomposition technique to establish a Yaglom-type limit for branching processes in random environment, i.e., when the environment is given by a sequence of random probability measures on $\n_0$. If the random environment is an i.i.d. sequence of probability measures, the Yaglom-type limit theorem under a quenched approach  is known in the literature \cite[Theorem 6.2]{kersting2017discrete}. In particular, they showed that when the environment is given by linear fractional distributions, the Yaglom-type limit is  an exponential random variable.  Then, for these and other distributions the construction has to be the same but, for the two genealogical lines, one has to find the distribution of the generation of their most recent common ancestor that makes the method work. Furthermore, by using the approach of several spines decomposition it would be possible to study the genealogy of Galton-Watson processes in varying environment. For the moment this technique has only been done  in the constant environment case (see \cite{harris2020coalescent}). These possible applications highlights the potential and relevance of our methodology.

The remainder of the paper is organised as follows. In  Section 2, we   introduce the one-spine and two-spine decompositions. With this in hand, we  give an intuitive explanation of the result and we explain why the limit must be exponential. In Section 3, we give some properties of the  measures associated with these decompositions and we characterise them via their Laplace transform. Finally,  Section 4 contains the proof.

\section{Outline of the proof}\label{sec_outline}
In this section, we  provide an intuitive explanation of the result and  explain why the limit must be an exponential random variable.  First, we explain the one-spine and two-spines decompositions. Then, we  relate them with a size-biased characterisation of the exponential random variable.

Recall that given a random variable \ $X$\  and a Borel function \ $g$\  such that  \ $\p(g(X)\geq 0)=1$, \ and \ $\e[g(X)]\in (0,\infty)$, \ we say that \ $W$ \ is a \ {\it  $g(X)$-transform} of \ $X$ \ if 
\begin{equation*}
\e[f(W)] = \frac{\e[f(X)g(X)]}{\e[g(X)]},
\end{equation*}
for each positive Borel function \ $f$.\  If \ $g(x)=x$, \ we also call it the { \it size-biased transform. }

Observe that the law of a non-negative random variable \ $X$ \ conditioned on being strictly positive can be described in terms of its size-biased transform. More precisely, for each \ $\lambda\geq 0$,
\begin{equation}\label{conditionated}
\e\left[1-e^{-\lambda X}\mid X>0 \right] =\int_0^{\lambda} \frac{\e\left[Xe^{-sX}\right]}{\p(X>0)}\ud s= \e\left[X \mid X > 0\right] \int_{0}^{\lambda} \e\left[e^{-s\dot{X}}\right]\ud s,
\end{equation}
where \ $\dot{X}$ \ is the size-biased transform of \ $X$. \ Recall that a sequence of non-negative random variables converges in distribution if and only if their Laplace transforms converge. As a consequence, we obtain the following lemma 

\begin{lem}\label{remark 1}
	Let $\{X_n: n\geq 0\}$ be a sequence of non-negative random variables. Then 
	the variables conditioned on being strictly positive \ $\{X_n\ ;  \p(\cdot\mid X_n>0)\}_{n\geq 0}$  \ converge in distribution to a strictly positive random variable \ $Y$ \ if and only if \ $\Expcon{X_n}{X_n>0}\rightarrow \Exp{Y}$ \  and \ $\dot{X}_n$ \ converges in distribution to \ $\dot{Y}$, \ where \ $\dot{X}_n$ \  and \ $\dot{Y}$ \ are the size-biased transforms of \ $X_n$ \ and \ $Y$, respectively.
\end{lem}

By Lemma \ref{remark 1}, in order to prove Theorem  \ref{teo_yaglom} we need to study the size-biased process \ $\dot{Z}^Q:=\{\dot{Z}_n^Q:n\geq 0\}$. \ Recall that there is a relationship between Galton-Watson processes in environment \ $Q$ \ and Galton-Watson trees in environment \ $Q$. \
In the tree, any particle or individual in generation \ $i$ \ gives birth to particles in generation \ $i+1$ \ according to \ $q_{i+1}$. \ The variable \ $Z_n^Q$ \ is the number of particles at generation \ $n$ \ in the tree.  \ In a similar way,  \ $\dot{Z}^Q_n$\ is the population size at generation \ $n$ \ of some random tree.\ 
According to Kersting and Vatutin \cite[Sections 1.4.1 and 1.4.2]{kersting2017discrete}, the tree associated to \ $\dot{Z}^Q$ \ is a size-biased tree in varying environment \ $Q$. \ More precisely, for each $i\geq 1$, let  \ $\dot{q}_i$ \ be the size-biased transform of \ $q_i$, 
\begin{equation}\label{defqdot}
\dot{q}_{i}(k)= \frac{k}{f'_{i}(1)}q_{i}(k),\qquad k \in \n_0.
\end{equation}

The {\it size-biased tree in environment} \ $Q$\ is constructed as follows:
\begin{enumerate}
	\item[(i)] We first establish an initial marked particle,
	\item[(ii)] the marked particle in generation \ $i \in \n_0$ \
	gives birth to particles in generation \ $i+1$ \ according to \ $\dot{q}_{i+1}$. \  Uniformly, we select one of these particles as the  marked particle. All the others particles  are unmarked, 
	\item[(iii)] any unmarked particle in generation \ $i \in \n_0$ \ gives birth to unmarked particles in generation \ $i+1$ \ according to \ $q_{i+1}$, \ independently of other particles.
\end{enumerate}

The marked genealogical line is called {\it spine}. This construction is known as the {\it one-spine decomposition}; see Figure \ref{fig:1} below. The constant environment case was done by Lyons, Pemantle and Peres \cite{lyons1995conceptual}.   
According to Kersting and Vatutin, \ $\dot{Z}^Q_n$ \ is the number of particles at generation \ $n$ \ in this tree. 

\begin{figure}[!htbp]
	\begin{subfigure}[b]{0.38\textwidth}
		\includegraphics[width=\textwidth]{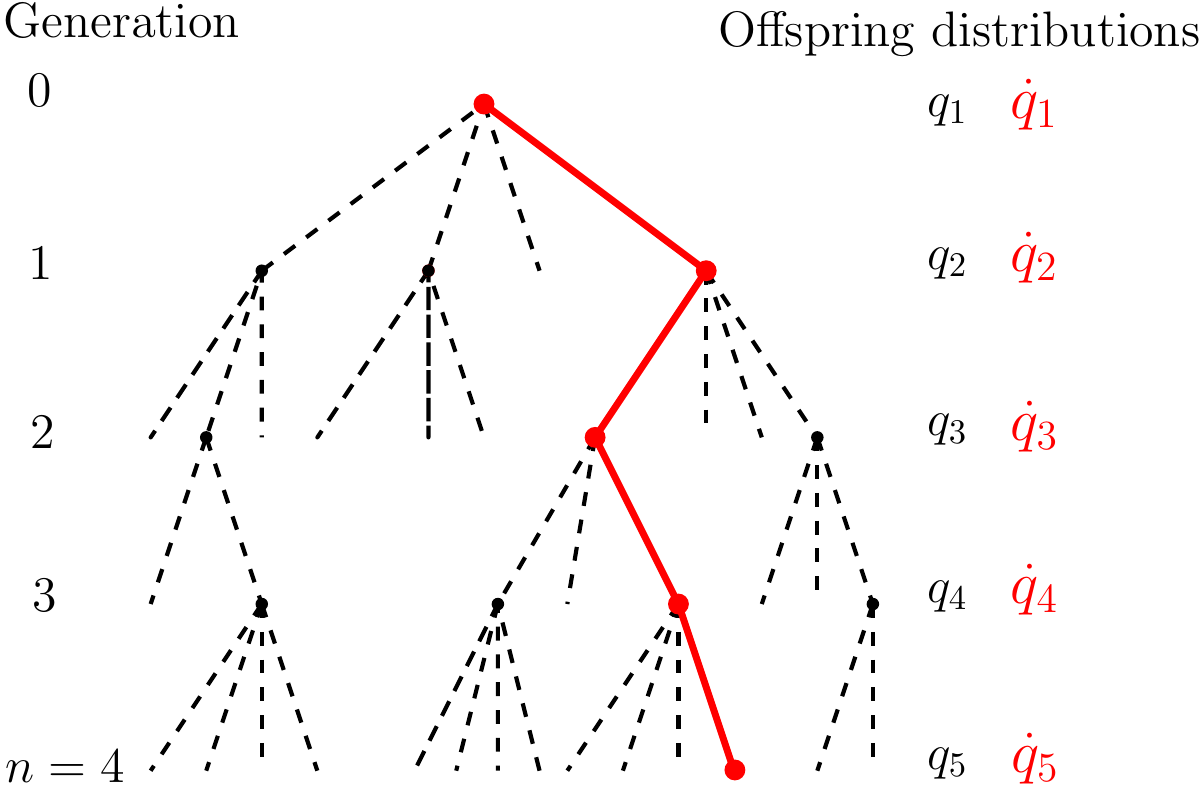}
		\caption{One-spine decomposition}
		\label{fig:1}
	\end{subfigure}
	\hspace{1.5cm}
	\begin{subfigure}[b]{0.38\textwidth}
		\includegraphics[width=\textwidth]{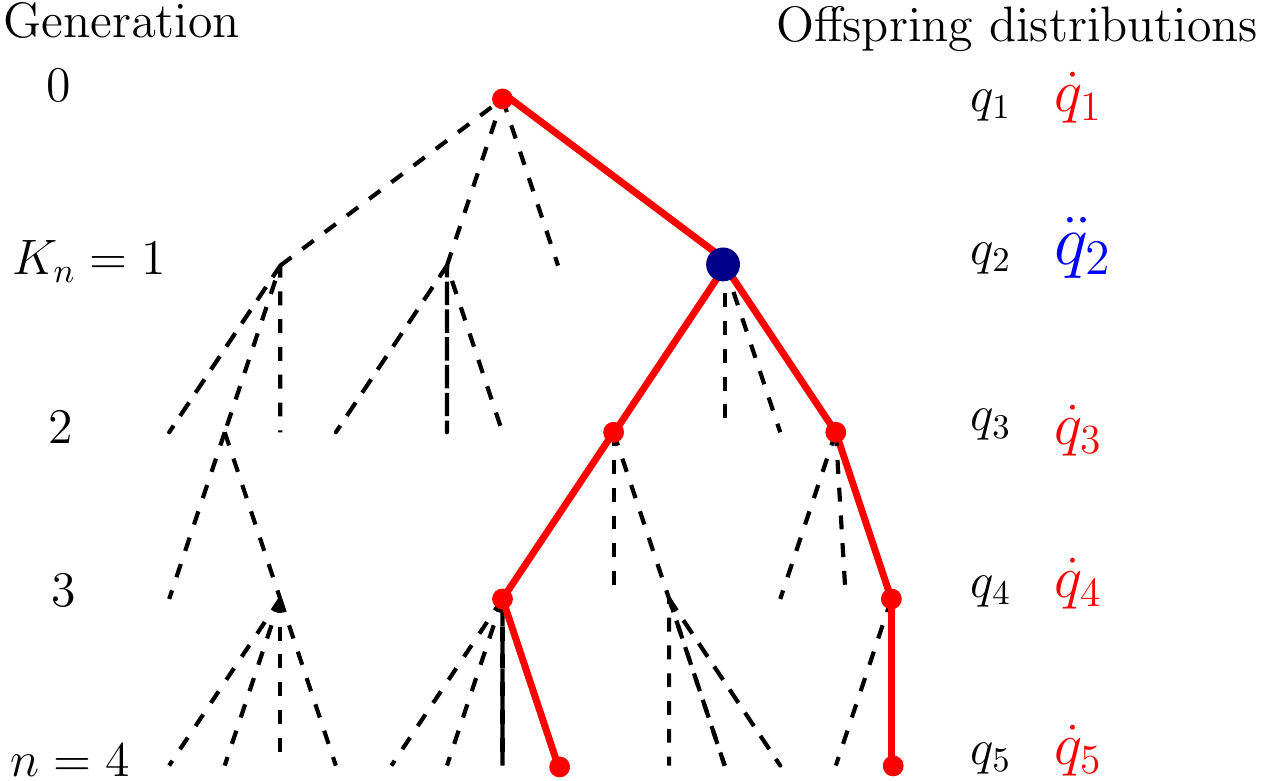}
		\caption{Two-spine decomposition}
		\label{fig:2}
	\end{subfigure}
	\caption{Spine decompositions}
	\label{fig:01}
\end{figure}

Now, we want to construct a random tree up to generation \ $n$ \ with two marked genealogical lines  or spines.  Denote by \ $K_n$ \ the {\it  generation of the most recent common ancestor of the lines}. Note that before \ $K_n$ \ there is only one spine and in generation \ $K_n+1$ \  a second spine is created. Since the offspring  distribution is varying among generations,  \ $K_n$ \ should depend on the environment. We assume that in this construction, \ $K_n$ \ has the following  distribution 
\begin{equation}\label{eq_ley_Kn}
\p(K_n= r) := \frac{\nu_{r+1}}{\mu_{r}}
\left(\sum_{k=0}^{n-1}\frac{\nu_{k+1}}{\mu_{k}}\right)^{-1},
\qquad 0 \leq r \leq n-1,
\end{equation}
where \ $\mu_n$ \ and \ $\nu_n$ \ are defined in \eqref{defi mu}. Thus, by \eqref{defi mu}, generations with larger offspring mean or larger offspring variance are more probably to be chosen as \ $K_n$. In generation $K_n$, we need to have an offspring distribution with two or more individuals.  We denote by \ $\ddot{q}_i$ \ the \ $q_i(q_i-1)$-transform of \ $q_i$ \ given by
\begin{equation}\label{defqddot}
\ddot{q}_{i}(k) = \frac{k(k-1)q_{i}(k)}{\nu_{i} f'_i(1)^2}, 
\qquad k \in \n_0, \quad  i=1,\dots ,n.
\end{equation}

We define a \textit{\ $X(X-1)$-type size-biased  tree in environment \ $Q$ \ up to time \ $n$\ } as the tree constructed as follows:
\begin{enumerate}
	\item[(i)] we first establish an initial marked particle,
	\item[(ii)] select  \ $K_n$ \ according to \eqref{eq_ley_Kn},
	\item[(iii)] the marked particle in generation \ $K_n$ \ gives birth to particles according to  \ $\ddot{q}_{K_n+1}$. \ Uniformly without replacement, we select two of these particles as the marked particles in generation \ $K_n+1$. \ The other particles are unmarked,
	\item[(iv)] any marked particle in   generation \ $i \in \{0,\ldots,n-1\}\setminus K_n$ \ gives birth to particles in generation \ $i+1$ \ according to \ $\dot{q}_{i+1}$. \ Uniformly, select one of these as the marked particle. All the other particles are not marked,
	\item[(v)] any unmarked particle in  generation \ $i \in \{0,\ldots,n-1\}$ \ gives birth to unmarked particles in generation $i+1$ according to \ $q_{i+1}$,  independently of other particles.
\end{enumerate}

We call this construction as the {\it two-spine decomposition}; see Figure \ref{fig:2}. Ren et. al \cite{ren20172} provided a two spine decomposition for Galton-Watson processes in a constant environment. In this case, the distribution of \ $K_n$ \ is uniform in \ $\{0,\ldots, n-1\}$. \ Using that the environment is constant we can recover their construction.  

With these constructions, we can give an intuitive explanation of why the limit must be an exponential random variable, we will make this intuition rigorous in the following sections. For any \ $0\leq k\leq n$, \ let  \ $\ddot{Z}^Q_k$ \ be the population size at the \ $k$-th generation in the previous tree. From the constructions of the size-biased trees (see Figure \ref{fig:01}), we see that we can decompose the particles associated to \ $\ddot{Z}^Q_n$ \ into descendants attached to the longer spine and descendants attached to the shorter spine. The descendants attached to the longer spine can be seen as the population in the \ $n$-th generation of a size-biased tree with environment \ $Q$, \ while the descendants of the shorter spine  are approximately distributed  as the population in  generation \ $n-(K_n+1)$ \ of a size-biased tree with environment \ $Q_{K_n+1}:=(q_{K_n+2},q_{K_n+3},\ldots)$. \  By construction, the two subpopulations are independent.
Therefore, we have roughly that
\begin{equation}
\ddot{Z}_n^Q \stackrel{(d)}{\approx} \dot{Z}_n^Q+ \dot{Z}^{Q_{K_n+1}}_{n-(K_n+1)},\qquad n\geq 1, \label{aprox}
\end{equation}
where the right-hand side of the equation is an independent sum. 
If we normalise with \ $a_n^Q$, \ we obtain
\begin{equation}\label{eq Z_n/a_n}
\frac{\ddot{Z}_n^Q}{a_n^Q} \stackrel{(d)}{\approx} \frac{\dot{Z}_n^Q}{a_n^Q} + \frac{a^{Q_{K_n+1}}_{n-(K_n+1)}}{a_n^Q}\frac{\dot{Z}^{Q_{K_n+1}}_{n-(K_n+1)}}{a^{Q_{K_n+1}}_{n-(K_n+1)}},\qquad n\geq 1.
\end{equation}

Kersting and Vatutin \cite[Lemma 1.2]{kersting2017discrete} proved that \ $\dot{Z}^Q_n$ \ is the size-biased transform of \ $Z_n^Q$. \ In this paper, we provide a precise meaning of equation \eqref{aprox}, we prove that \ $\ddot{Z}_n^Q$ \ is the \ $Z_n^Q(Z_n^Q-1)$-transform of \ $Z_n^Q$ \ (see Proposition \ref{teo_2}), that   
$$(a_n^Q)^{-1}a^{Q_{K_n+1}}_{n-(K_n+1)}\stackrel{(d)}{\longrightarrow} U,\qquad  \mbox{ as } n\rightarrow\infty,$$
where $U$ is an uniform random variable on $[0,1]$\ (see Proposition \ref{Prop_unifor}), and that \ $\dot{Z}^Q_n/a_n^Q$ \ converges in distribution to a random variable \ $\dot{Y}$ (see Proposition \ref{limit Zdot}).

Since  \ $\ddot{Z}_n^Q$ \ is the  $(\dot{Z}_n^Q-1)$-transform of \ $\dot{Z}_n^Q$, we have that $\ddot{Z}^Q_n/a_n^Q$ \ converges in distribution to  \ $\ddot{Y}$, \ the $\dot{Y}$-transform of $\dot{Y}$. 
Hence, by Lemma \ref{remark 1},  if we take limits in \eqref{eq Z_n/a_n}, we see that  \ $Z^Q_n/a_n^Q$ \ conditioned on being strictly positive converges in distribution to a random variable $Y$ that satisfies
\begin{equation}\label{eq_expone}
\ddot{Y} \stackrel{(d)}{=} \dot{Y} + U  \cdot \dot{Y}'
\end{equation}
where \ $\dot{Y}$ \ and \ $\dot{Y}'$\  are both \ $Y$-transforms of \ $Y$,\  \ $\ddot{Y}$\ is a \ $Y^2$-transform of \ $Y$,\ and  \ $U$ \ is an uniform random variable on \ $[0,1]$ \ independent of \ $\dot{Y}$ \and \ $ \dot{Y}'$. Ren et. al. \cite[Lemma 1.3]{ren20172}, showed that a variable \ $Y$ \ is exponentially distributed with mean 1 if and only if \eqref{eq_expone} holds. Therefore, 
\ $Z^Q_n/a_n^Q$ \ must converge in distribution to a standard exponential random variable.

\section{Size-biased trees}
In this section, we  study the size-biased trees defined in the previous section. We associate them to probability measure in the set of  rooted trees. For this purpose, we introduce the so-called {\it Ulam-Harris labeling}. 
Let \ $\mathcal{U}$ \ be the set of finite sequences of strictly positive integers, including \ $\emptyset$. \ For \ $ u\in \mathcal{U}$, \ we define the length of \ $u$ \ by \ $|u|:=n$, if  \ $u=u_1\cdots u_n$, where $n\geq 1$ \ and by \ $|\emptyset|:=0$ if $u=\emptyset$.  \ If \ $u$ \ and \ $v$ \ are two elements in \ $\mathcal{U}$,\ we denote by \ $uv$ \ the concatenation of \ $u$ \ and \ $v$,\ with the convention  that \ $uv=u$ \ if \ $v=\emptyset$. \ The genealogical line of \ $u$ \ is denoted by \ $[\emptyset,u]=\{\emptyset\}\cup\{u_1\cdots u_j: j=1,\ldots ,n\}$.\ Let  \ $\textbf{s}\ \subset\ \mathcal{U}$, \ its most recent common ancestor is the unique element \ $v\in \cap_{u\in\textbf{s}}[\emptyset, u]$ \ with maximal length and its generation is denoted by \  $K_{\textbf{s}}$. 

A {\it rooted tree} \ $\textbf{t}$ \ is a subset of \ $\mathcal{U}$ \ that satisfies  \ $\emptyset\in\textbf{t}$,\  $[\emptyset, u]\subset\textbf{t}$ for any $u\in\textbf{t}$, \ and if \ $u\in \textbf{t}$ \ and \ $i\in\n$ \ satisfy that \ $ui\in\textbf{t}$ \ then, \ $uj\in\textbf{t}$ \ for all \ $1
\leq j\leq i$. \ Denote by \ $\mathcal{T} = \{\textbf{t}: \textbf{t}\ \mbox{is a tree}\},$ the subspace of rooted trees. \ The vertex \ $\emptyset$ \ is called {\it the root} of the tree. For any \ $u\in \textbf{t}$, \ we define the number of offspring of \ $u$ \ by \ $l_u(\textbf{t})=\max\{i\in\z^+: ui\in\textbf{t}\}$. \ The height of \ $\textbf{t}$ \ is defined by \ $|\textbf{t}|= \sup\{|u|: u \in \textbf{t}\}$. \ For any \ $n\in \n$ \ and \ $\textbf{t}, \tilde{\textbf{t}}$ \ trees, we write \ $\textbf{t}\overset{n}{=}\tilde{\textbf{t}}$ \ if they coincide up to height  \ $n$. \ The population size in the \ $n$-th generation of the tree \ $\textbf{t}$ \ is denoted by \ $X_n(\textbf{t})=\#\{u\in\textbf{t}: |u|=n\}$.

A {\it Galton-Watson tree in the environment} \ $Q=(q_1,q_2,\ldots)$ \ is a \ $\mathcal{T}$-valued random variable \ $\textbf{T}$ \ such that 
$$\textbf{G}_n(\textbf{t}):= \p(\textbf{T}\overset{n}{=}\textbf{t}) = \prod_{u\in \textbf{t}:\, |u|<n}  q_{|u|+1}(l_u(\textbf{t})),$$
for any \ $n\geq 0$\  and any tree \ $\textbf{t}$. \ As we said before, the process \ $Z=\{Z_n^Q:n\geq 0\}$ \ defined as \ $Z_n^Q=X_n(\textbf{T})$ \ is a Galton-Watson process in environment \ $Q$. 

Now, we deal with the one-spine  decomposition. This construction builds a tree along a distinguished path.  More precisely, a {\it spine} or distinguished path \ $\textbf{v}$ \ on a tree \ $\textbf{t}$ \ is a sequence \  $\{v^{(k)}: k=0,1,\ldots,|\textbf{t}|\}\subset \textbf{t}$ \ (or $\{v^{(k)}: k=0,1,\ldots\}\subset \textbf{t}$ if $|\textbf{t}|=\infty$) such that \ $v^{(0)}=\emptyset$\ and  $v^{(k)}=v^{(k-1)}j$ \ for some \ $j\in\n$, \ for any \ $1\leq k\leq |\textbf{t}|$. \  We denote by $\dot{\mathcal{T}}$, the {\it subspace of trees with one spine} 
$$\dot{\mathcal{T}}= \{(\textbf{t},\textbf{v}):\ \textbf{t}\ \mbox{is a tree and } \ \textbf{v}\ \mbox{is a spine on}\ \textbf{t} \}$$ 
and by \ $\mathcal{T}_n=\{\textbf{t}\in \mathcal{T}: |\textbf{t}|=n\}$ \ and \ $\dot{\mathcal{T}}_n=\{(\textbf{t},\textbf{v})\in \dot{\mathcal{T}}: |\textbf{t}|=n\}$ \ the restriction of \ $\mathcal{T}$ \ and \ $\dot{\mathcal{T}}$ \ to trees with height \ $n$.\ 

We are going to construct the probability distribution of the size-biased tree in the environment \ $Q$ on the state space $\mathcal{T}$. \ First, we need to define a probability distribution on \ $\dot{\mathcal{T}}$.  Recall the construction of the size-biased tree in the previous section; individuals  along the spine, $\{u\in \textbf{t}: u\in \textbf{v}\}$, have offspring distribution \ $\dot{q}_{|u|+1}$ \ given by \eqref{defqdot}, and from their offspring we select one uniformly as the spine individual in the next generation.  Individuals outside the spine, $\{u\in \textbf{t}: u\notin \textbf{v}\}$, have offspring distribution \ $q_{|u| +1}$. \ Then, the size-biased tree can be seen as a \ $\dot{\mathcal{T}}$-valued random variable \ $ (\dot{\textbf{T}},\textbf{V})$ \ with distribution
$$ \p((\dot{\textbf{T}},\textbf{V})\overset{n}{=}(\textbf{t},\textbf{v})) 
:= \prod_{u\in \textbf{v}:\, |u|<n}  \dot{q}_{|u|+1}(l_u(\textbf{t}))\frac{1}{l_u(\textbf{t})}\prod_{u\in \textbf{t}\setminus\textbf{v}:\, |u|<n}  q_{|u|+1}(l_u(\textbf{t})),$$
for any \ $n\geq 0$ \ and any \ $(\textbf{t},\textbf{v})\in\dot{\mathcal{T}}_n$. \ One readily checks that this measure is a probability  on \ $\dot{\mathcal{T}}$ \ by using  the definition of \ $\dot{q}$ \ and the fact that \ $\textbf{G}_n$ \ is a probability measure. In a similar way, we can write  
$$\p((\dot{\textbf{T}},\textbf{V})\overset{n}{=}(\textbf{t},\textbf{v})) =\frac{1}{\mu_n}\cdot \textbf{G}_n(\textbf{t}), \quad (\textbf{t},\textbf{v})\in \dot{\mathcal{T}}.$$
Hence, by summing over all the possible spines, we obtain the distribution of the {\it size-biased Galton-Watson  tree  in environment} \ $Q$ \ on \ $\mathcal{T}$ 
\begin{equation*}\label{eq_med_Gdot}
\dot{\textbf{G}}_n(\textbf{t}):=\p(\dot{\textbf{T}}\overset{n}{=}\textbf{t})=  \sum_{\textbf{v}: (\textbf{t},\textbf{v})\in\dot{\mathcal{T}}_n} \p((\dot{\textbf{T}},\textbf{V})\overset{n}{=}(\textbf{t},\textbf{v})) =  \frac{1}{\mu_n} X_n(\textbf{t})\cdot \textbf{G}_n(\textbf{t}),
\end{equation*}
for any \ $n\geq 0$ \ and any \ $\textbf{t}\in\mathcal{T}_n$  (see also \cite[Lemma 1.2]{kersting2017discrete}). \  Define the process \ $\dot{Z}^Q=\{\dot{Z}_n^Q: n\geq 0 \}$ \ as 
\ $\dot{Z}_n^Q=X_n(\dot{\textbf{T}})$, \ for each \ $n\geq 1$. \ 
Then, by using the measure \ $\dot{\textbf{G}}_n$ \ we can see that the process \ $\{\dot{Z}_m^Q: 0\leq 
m\leq n\}$ \ is a \ $Z_n^Q$-transform of  \ $\{Z_m^Q: 0\leq m\leq n\}$, \ in other words
$$\Exp{g(\dot{Z}_1^Q,\ldots,\dot{Z}_n^Q)}
=\frac{\Exp{Z_n^Qg(Z_1^Q,\ldots,Z_n^Q)}}{\Exp{Z_n^Q}}, \qquad \mbox{ for all bounded functions } g.$$

Now we consider  the probability distribution associated to the \ $X(X-1)$-type size-biased tree up to time \ $n$ \ on the state space \ $\mathcal{T}_n$. \ As we did before, we define a measure on 
$$\ddot{\mathcal{T}}_n:
= \left\{(\textbf{t},\textbf{v}, \tilde{\textbf{v}}): \ (\textbf{t}, \textbf{v}), (\textbf{t}, \tilde{\textbf{v}}) \in \dot{\mathcal{T}}_n,\ \textbf{v}{\not=}\tilde{\textbf{v}}\right\}, \qquad n\in \n,$$
the {\it subspace of trees with height \ $n$ \ and two different spines}. 
Given a \ $(\textbf{t},\textbf{v}, \tilde{\textbf{v}})\in \ddot{\mathcal{T}}_n$, \ we denote by \ $K_{\textbf{v}, \tilde{\textbf{v}}}=\max\{r<n:\textbf{v}\overset{r}{=} \tilde{\textbf{v}}\}$ \ the generation of the most recent common ancestor of \ $\textbf{v}\cup\tilde{\textbf{v}}$. 

Recall the construction of a \ $X(X-1)$-type size-biased  tree in the previous section; (i) consider an initial spine individual, (ii) select the generation of the most recent common ancestor, $K_{\textbf{v}, \tilde{\textbf{v}}}$, according to \eqref{eq_ley_Kn}, (iii) the spine individual $u$ in that generation has offspring distribution \ $\ddot{q}_{|u|+1}$ \ given by \eqref{defqddot}. From its offspring we select uniformly without replacement two as spine individuals in the next generation, (iv) the spine individuals  in the other generations, $\{u\in \textbf{v}\cup \tilde{\textbf{v}}: |u|\neq K_{\textbf{v}, \tilde{\textbf{v}}}\}$, have offspring distribution \ $\dot{q}_{|u|+1}$ \ given by \eqref{defqdot}. From its offspring we select uniformly one as the spine individual in the next generation,  (v) finally, individuals  outside the spine, $\{u\in \textbf{t}: u \notin\textbf{v}\cup \tilde{\textbf{v}}\}$, have offspring distribution \ $q_{|u|+1}$. \ Then, the \ $X(X-1)$-type size-biased  tree up to time \ $n$ \ can be seen as a  \ $\ddot{\mathcal{T}}_n$-valued random variable \ $(\ddot{\textbf{T}},\textbf{V},\widetilde{\textbf{V}})$ \ with distribution 
\begin{equation*}
\begin{split}
\p((\ddot{\textbf{T}},\textbf{V},\widetilde{\textbf{V}})\overset{n}{=}(\textbf{t},\textbf{v},\tilde{\textbf{v}})) 
:{=}& \frac{\nu_{K_{\textbf{v}, \tilde{\textbf{v}}}+1}}{\mu_{K_{\textbf{v}, \tilde{\textbf{v}}}}} \left(\sum_{k=0}^{n-1}\frac{\nu_{k+1}}{\mu_k}\right)^{-1}\hspace{-.4cm}
\prod_{u\in \textbf{v}\cup \tilde{\textbf{v}}:\, K_{\textbf{v}, \tilde{\textbf{v}}}= |u|}  \ddot{q}_{|u|+1}(l_u(\textbf{t}))\frac{2}{l_u(\textbf{t})(l_u(\textbf{t} )-1)} \\
&\prod_{u\in \textbf{v}\cup \tilde{\textbf{v}}:\, K_{\textbf{v}, \tilde{\textbf{v}}}\neq |u|<n}  \dot{q}_{|u|+1}(l_u(\textbf{t}))\frac{1}{l_u(\textbf{t})}
\prod_{u\in \textbf{t}\setminus(\textbf{v}\cup \tilde{\textbf{v}}):\, |u|<n}  q_{|u|+1}(l_u(\textbf{t})),
\end{split}
\end{equation*}
for any \ $(\textbf{t},\textbf{v}, \tilde{\textbf{v}})\in \ddot{\mathcal{T}}_n$. \ Here, the first two terms in the right-hand side of the equation are associated with step (ii).  \ The first product is associated with step (iii). Then, in the second line, the first product is obtained with (iv). Finally, we use (v) to obtain the last product. By using the definition of \  $q$, \ $\dot{q}$ \ and \ $\ddot{q}$, \ one can readily verify that the previous expression defines a probability measure on \ $\ddot{\mathcal{T}}_n$. \ Moreover,  we have 
$$\p((\ddot{\textbf{T}},\textbf{V},\widetilde{\textbf{V}})
\overset{n}{=} (\textbf{t},\textbf{v}, \tilde{\textbf{v}} ) )
=\frac{2}{\mu_n^2} \left(\sum_{k=0}^{n-1} \frac{\nu_{k+1}}{\mu_k} \right)^{-1} \textbf{G}_n(\textbf{t}),$$
for any \ $(\textbf{t},\textbf{v}, \tilde{\textbf{v}})\in \ddot{\mathcal{T}}_n$. \
Then, by summing over all the possible two spines, we obtain that the \ $X(X-1)$-type size-biased  tree up to time $n$ is a  \ $\mathcal{T}_n$-valued random variable \ $\ddot{\textbf{T}}$ \ with law
\begin{equation}\label{eq_leq_G2p}
\ddot{\textbf{G}}_n(\textbf{t}):= \p(\ddot{\textbf{T}}\overset{n}{=}\textbf{t})
=\frac{1}{\mu_n^2} \left(\sum_{k=0}^{n-1} \frac{\nu_{k+1}}{\mu_k}\right)^{-1} X_n(\textbf{t})(X_n(\textbf{t}) -1)\cdot \textbf{G}_n(\textbf{t}),
\end{equation}
for any \ $\textbf{t}\in\mathcal{T}_n$.  Define the process \ $\ddot{Z}^Q=\{\ddot{Z}_m^Q: 0\leq m \leq n\}$ \  by \ $\ddot{Z}_m^Q=X_m(\ddot{\textbf{T}}).$

Opposite to what happens with \ $(\dot{\textbf{G}}_n:n\geq 1)$, \ by construction,  the measures \ $(\ddot{\textbf{G}}_n:n\geq 1)$ \ are not consistent in the sense that \ $\ddot{\textbf{G}}_n$ \ is not a restriction of \  $\ddot{\textbf{G}}_{n+1}$ \ to the tree with size \ $n$. \    More precisely, in the size-biased tree, the change of measure is intuitively a martingale since the tree under this measure has one spine throughout all generations. While, in  the \ $X(X-1)$-type size-biased  tree, if we restrict a tree with two spines at time \ $n$ \ to the previous generations it is possible to lose one spine. Indeed, the tree will have only one marked particle in all the generations before \ $ K_{\textbf{v}, \tilde{\textbf{v}}}$.\  Then, the change of measure in the next proposition is not a martingale change of measure,  not even in the case of constant environment  \cite[Theorem 1.2]{ren20172}. However, it allows us to conclude that  \ $\{\ddot{Z}^Q_m:0\leq m\leq n\}$ \ is a \ $Z_n^Q(Z_n^Q-1)$-transform of \ $\{Z_m^Q: 0\leq m\leq n\}$. 
\begin{proposition}\label{teo_2}
	Let \ $\{Z_n^Q: n\geq 0\}$ \  be a GWVE and for any \ $n\in \n_0$, \ let  \ $\ddot{Z}^Q=(\ddot{Z}_m^Q: 0\leq m \leq n)$ \ be the process associated with the $X(X-1)$-type size-biased  tree up to time \ $n$. \ Then, for any bounded function \ $g:\z^{n}_+\rightarrow \r$, \
	\begin{eqnarray}
	\e[g(\ddot{Z}_1^Q,\dots, \ddot{Z}_n^Q)] = \frac{\e[Z_n^Q(Z_n^Q -1) g(Z_1^Q,\dots,Z_n^Q)]}{\e[Z_n^Q(Z_n^Q-1) ]}.
	\end{eqnarray}
\end{proposition}

\begin{proof}
	Fix \ $n\geq 0$ \ and recall that for each \ $m\leq n$, \ $Z_m^Q=X_m(\textbf{T})$  under the measure  $\textbf{G}_n$\ and \ 
	$\ddot{Z}_m^Q=X_m(\textbf{T})$ under the measure $\ddot{\textbf{G}}_n$. \
	Hence, by \eqref{eq_leq_G2p}
	\begin{equation*}
	\begin{split}
	\e[g(\ddot{Z}_1^Q, \ldots ,\ddot{Z}_n^Q)] 
	&= \ddot{\textbf{G}}_n[g(X_1(\textbf{T}), \ldots , X_n(\textbf{T}) )] \\ 
	&= \frac{1}{\mu_n^2} \left(\sum_{k=0}^{n-1} \frac{\nu_{k+1}}{\mu_k}\right)^{-1} \textbf{G}_n\left[ X_n(\textbf{T})(X_n(\textbf{T}) -1)g(X_1(\textbf{T}), \ldots , X_n(\textbf{T}) ) \right] \\ 
	&= \frac{1}{\mu_n^2} \left(\sum_{k=0}^{n-1} \frac{\nu_{k+1}}{\mu_k}\right)^{-1} \e\left[ Z_n^Q(Z_n^Q -1)g(Z^Q_1, \ldots, Z^Q_n) \right].
	\end{split}
	\end{equation*}
	By taking \ $g\equiv1$, \ we deduce that 
	$$  \e[ Z^Q_n(Z^Q_n -1)] = \mu_n^2 \sum_{k=0}^{n-1} \frac{\nu_{k+1}}{\mu_k},$$ 
	which implies the result.
\end{proof}

In the reminder of this section, we study some properties of the previous decompositions. We first introduce the notation to refer to shifted environments. Let \ $q$ \ be a probability measure on \ $\n_0$ \ such that \ $q(\{0,1,\ldots, r-1\})=0$ for some $r\in \n$.\ We define the probability measure  \ $[q-r]$ \ in \ $\n_0$ \ by \ $[q-r](i)=q(i+r)$ \ for all \ $i\in\n_0$. \  Given a probability measure \ $q$ \ and an environment \ $Q=(q_1,q_2,...)$, \ we denote
$$q\oplus Q:=(q,q_1,q_2,\ldots).$$
For any \ $m\in\n_0$, \ as in Section \ref{sec_outline} we set
$$Q_m:= (q_{m+1}, q_{m+2},...).$$

We can compute the Laplace transform of \ $\ddot{Z}_n^Q$ \ in terms of the Laplace transform of \ $\dot{Z}_n^Q$ \  and \ $ \dot{Z}_{n-(m+1)}^{Q_{m+1}}$, \ as indicated below. The  proof follows similar arguments as those used in  \cite[Proposition 2.1]{ren20172}, although the presence of varying environment leads to significant changes.

\begin{proposition}\label{prop_desco_Zdosptos}
	Fix \ $n\geq 1$. \ 
	Let  \  $\{\dot{Z}_m^{\cdot}: m\leq n\}$ \  and \ $\{\ddot{Z}_m^{\cdot}: m\leq n\}$ \ be the population size of the size-biased tree and the \ $X(X-1)$-type size-biased tree up to time \ $n$. \ Then, we have the following decomposition, for each \ $\lambda \geq 0$
	\begin{equation*}
	\e\left[\expo{-\lambda \ddot{Z}_n^Q}\right] 
	= \e\left[\expo{-\lambda \dot{Z}_n^Q } \right]  \sum_{m=0}^{n-1} \p(K_n = m) \e\left[\expo{-\lambda \dot{Z}_{n-(m+1)}^{Q_{m+1}}} \right] g(n,m, \lambda),
	\end{equation*}
	where the function \ $g$ \ is defined as follows
	\begin{equation}\label{eq_g}
	g(n,m, \lambda):= \frac{\e\left[\expo{-\lambda Z_{n-m}^{[\ddot{q}_{m+1}-2]\oplus Q_{m+1}}}\right]}{\e\left[\expo{-\lambda Z_{n-m}^{[\dot{q}_{m+1}-1]\oplus Q_{m+1}}}\right]},  
	\qquad 0\leq m \leq n-1, \quad 0\leq \lambda.
	\end{equation}
\end{proposition}

\begin{proof}
	Let	\ $\dot{\textbf{T}}$ \ be a size-biased Galton-Watson tree in environment \ $Q$ \ up to time \ $n$. \ We can decompose \ $\dot{\textbf{T}}$ \ into subtrees with roots along the spine \ $\textbf{V}$;\  see  Figure \ref{fig:3}.  More precisely, for every \ $0\leq k\leq n$, \ there is a \ $v^{(k)}\in \textbf{V}$ \ with \ $|v^{(k)}|=k$ \ and a random tree \ $\textbf{t}_k\in\mathcal{T}$ \ such that
	$$v^{(k)}\, \textbf{t}_k= \{u\in \dot{\textbf{T}}: | [\emptyset,u]\cap \textbf{V} |=k\} \qquad \mbox{ and }\qquad \dot{\textbf{T}}=\underset{k=0}{\overset{n}{\bigsqcup}}v^{(k)}\, \textbf{t}_k,$$
	where \ $\bigsqcup$ \ denotes the disjoint union. Note that \ $X_n(\dot{\textbf{T}})=\underset{k=0}{\overset{n} {\sum}}X_{n-k}(\textbf{t}_k).$ \  In the size-biased tree, each individual along the spine gives birth according to $\dot{q}_{\cdot}$ and one of its offspring is the spine individual in the next generation. Then, it follows that  the  subtrees \ $\textbf{t}_k,$ \ $0\leq k\leq n$, \ are independent Galton-Watson trees with environment \ $[\dot{q}_{k+1}-1]\oplus Q_{k+1}$. \ Therefore,
	\begin{equation}\label{eq_dist_1}
	\e\left[\expo{-\lambda \dot{Z}_n^Q}\right] 
	=  \prod_{k=0}^{n}\e\left[\expo{-\lambda Z_{n-k}^{[\dot{q}_{k+1}-1]\oplus Q_{k+1}}}\right], \qquad  \lambda \geq 0, \ n\in \n_0.
	\end{equation}

	\begin{figure}[!htbp]
		\begin{subfigure}[b]{0.38\textwidth}
			\includegraphics[width=4.5cm]{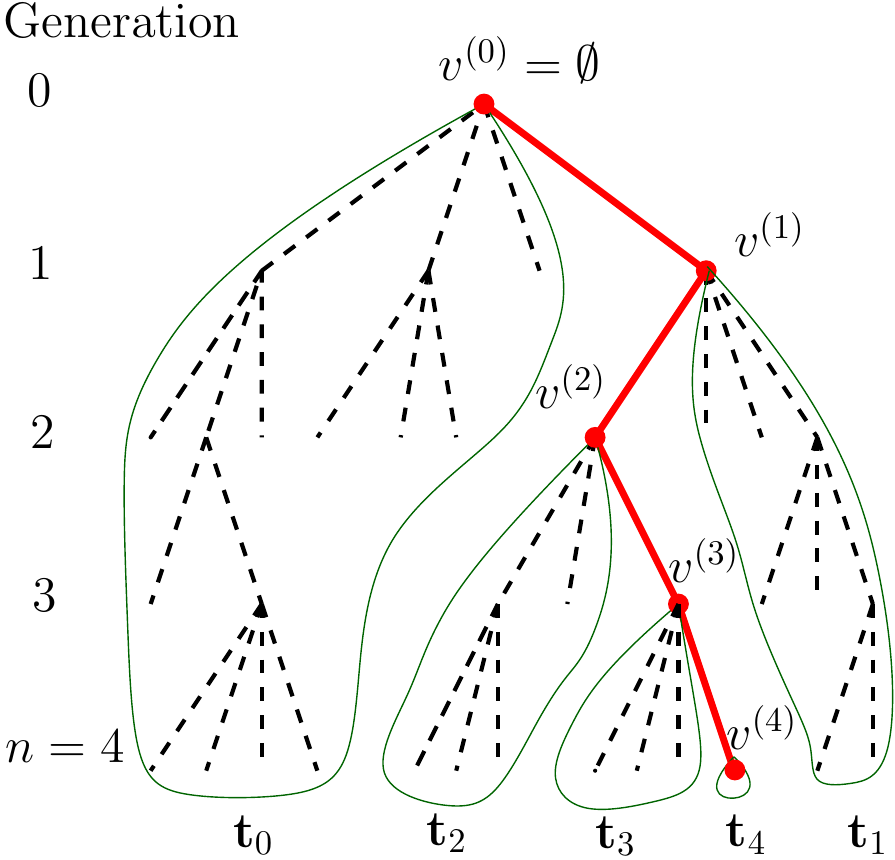}
			\caption{Size-biased tree}
			\label{fig:3}
		\end{subfigure}
		\hspace{1.5cm}
		\begin{subfigure}[b]{0.4\textwidth}
			\includegraphics[width=5cm]{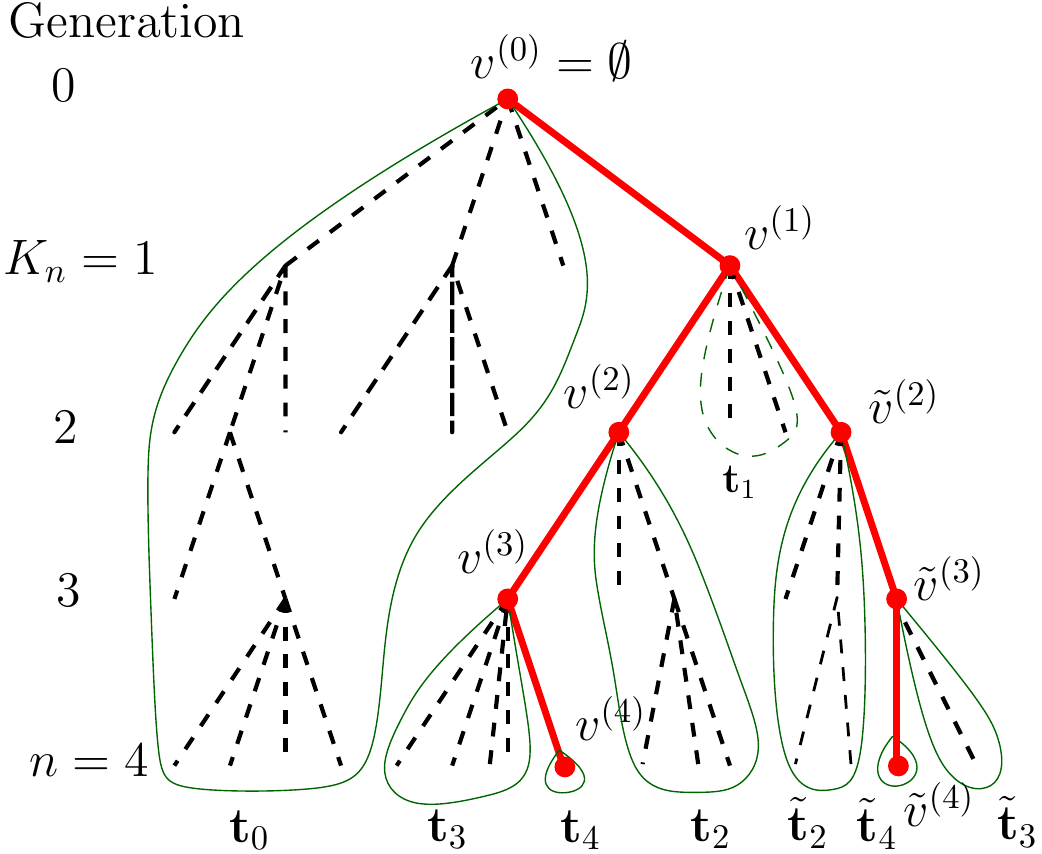}
			\caption{$X(X-1)$-type size-biased tree}
			\label{fig:4}
		\end{subfigure}
		\caption{Subtrees along the spine(s).}
	\end{figure}
	
	Let \ $\ddot{\textbf{T}}$ \ be a \ $X(X-1)$-type size-biased Galton-Watson tree up to time \ $n$. \ In a similar way, we can decompose	\ $\ddot{\textbf{T}}$ \ in subtrees with roots along the spines; see Figure \ref{fig:4}. Denote by \ $\textbf{V}$ \ and \ $\widetilde{\textbf{V}}$, \ the associated spines and recall that  \ $K_{n}=\max\{r<n: \textbf{V}\overset{r}{=}\widetilde{\textbf{V}}\}$. \ We can  form a partition of \ $\ddot{\textbf{T}}$ \ in the sense that
	\begin{equation}\label{partition}
	\ddot{\textbf{T}}=\left(\underset{k=0}{\overset{n}{\bigsqcup}}v^{(k)}\textbf{t}_k\right)\bigsqcup\left(\underset{k=1+K_n}{\overset{n}{\bigsqcup}}\tilde{v}^{(k)}\tilde{\textbf{t}}_k\right)
	\quad \mbox{ and }\quad 
	X_n(\ddot{\textbf{T}})=\underset{k=0}{\overset{n} {\sum}}X_{n-k}(\textbf{t}_k) +\underset{k=1+K_n}{\overset{n} {\sum}}X_{n-k}(\tilde{\textbf{t}}_k),
	\end{equation}
	where, for every \ $0\leq k\leq 
	K_n$, \  $v^{(k)}\in  \textbf{V}\cap\widetilde{\textbf{V}}$ \ and \ $\textbf{t}_k\in \mathcal{T}$ \ are such that \ $|v^{(k)}|=k$ \ and
	$$v^{(k)}\, \textbf{t}_k
	= \{u\in \ddot{\textbf{T}}: | [\emptyset,u]\cap (\textbf{V}\cup \widetilde{\textbf{V}})|=k\};$$
	and, for every
	\ $ K_{n}<k\leq n$,\  $v^{(k)}\in  \textbf{V}$,\  $\tilde{v}^{(k)}\in\widetilde{\textbf{V}}$\  and \ $ \textbf{t}_k, \tilde{\textbf{t}}_k\in \mathcal{T}$ \ satisfy  \ $|v^{(k)}|=k=|\tilde{v}^{(k)}|$, 
	$$v^{(k)}\, \textbf{t}_k= \{u\in \ddot{\textbf{T}}: | [\emptyset,u]\cap \textbf{V}|=k\}\qquad \mbox{ and }\qquad  \tilde{v}^{(k)}\, \tilde{\textbf{t}}_k= \{u\in \ddot{\textbf{T}}: | [\emptyset,u]\cap \widetilde{\textbf{V}}|=k\}.$$
	
	Observe that by the branching property, the subtrees are independent.   The spine individual at generation \ $K_n=m$ \ has offspring distribution $\ddot{q}_{m+1}$, and from its offspring we select two as the spine individuals in the next generation. Then  the subtree \
	$\mathbf{t}_m$ \ is a Galton-Watson tree with environment
	\ $[\ddot{q}_{m+1}-2]\oplus Q_{m+1}.$ \ The  other subtrees \ $\{\mathbf{t}_k: 0\leq k\leq n, k\neq m\}$ \ and \ $\{\tilde{\mathbf{t}}_k: m<k\leq n\}$ \ are Galton-Watson trees with environment \ $ [\dot{q}_{k+1}-1]\oplus Q_{k+1}$. \ Therefore, by using the  decomposition \eqref{partition}, we have 
	\begin{equation*}
	\begin{split}
	\e\left[\expo{-\lambda \ddot{Z}^Q_n}\right] 
	&= \sum_{m=0}^{n-1} \p(K_n= m) \e\left[\expo{-\lambda Z_{n-m}^{[\ddot{q}_{m+1}-2]\oplus Q_{m+1}}}\right]   \\ 
	\times&\prod_{k=0, k\not= m}^{n} \e\left[\expo{-\lambda Z_{n-k}^{[\dot{q}_{k+1}-1]\oplus Q_{k+1}}}\right]\prod_{k= m+1}^{n} \e\left[\expo{-\lambda Z_{n-k}^{[\dot{q}_{k+1}-1]\oplus Q_{k+1}}}\right] .
	\end{split}
	\end{equation*}
	Finally, if we apply equation \eqref{eq_dist_1} for environments \ $Q$ \ and \ $Q_{m+1}$, \ we obtain the result. In other words,
	\begin{equation*}
	\begin{split}
	\e\left[\expo{-\lambda \ddot{Z}^Q_n}\right] 
	=&
	\sum_{m=0}^{n-1} \p(K_n = m) \e\left[\expo{-\lambda Z_{n-m}^{[\ddot{q}_{m+1}-2]\oplus Q_{m+1}}}\right]\\ &\times\frac{\e\left[\expo{-\lambda \dot{Z}^Q_n}\right]}{\e\left[\expo{-\lambda Z_{n-m}^{[\dot{q}_{m+1}-1]\oplus Q_{m+1}}}\right]} \e\left[\expo{-\lambda \dot{Z}_{n-(m+1)}^{Q_{m+1}}} \right].
	\end{split}
	\end{equation*}
\end{proof}

The distribution of the previous processes can be expressed via the generating functions \ $(f_1,f_2,\ldots)$\  associated to \ $Q=(q_1,q_2,\ldots)$. \
For each \  $0\leq m \leq n$ \ and $s\in[0,1]$ we define 
$$ f_{m,n}(s):= [f_{m+1} \circ \ldots \circ f_n](s),$$ 
and \ $f_{n,n}(s):=s$. \  The generating function of \ $Z_n^Q$ \ is equal to \ $f_{0,n}$. \ For the others, we note that  for every \ $s\in[0,1]$ \ and \ $0\leq m < n$,
\begin{equation}\label{formula f''}
f_{m,n}'(s) = \prod_{l=m+1}^{n} f_l'(f_{l,n}(s)),\qquad \quad 
f_{m,n}''(s) = f_{m,n}'(s)^2\sum_{l=m+1}^{n} \frac{f_l''(f_{l,n}(s))}{f_l'(f_{l,n}(s))^2 \prod_{j=m+1}^{l-1} f_j'(f_{j,n}(s))},
\end{equation}
where \ $f_{n,n}'(s)=1$ \ and \ $f_{n,n}''(s)=0$.

\begin{lem}\label{prop_relat} 
	Let $n\geq 1$ \ and  \ $Q$ \ be a varying environment. Let \ $(Z_m^{\cdot} : 0\leq m \leq n)$, $(\dot{Z}_m^{\cdot}: 0\leq m \leq n)$ \ and \ $ (\ddot{Z}_m^{\cdot} : 0\leq m \leq n)$ \ be a GWVE, a sized-biased GWVE and a \ $X(X-1)$-type sized-biased GWVE up to time \ $n$. Then, for any \ $0\leq m< n$ \ and \ $ \lambda\geq 0$, 
	\begin{eqnarray}
	\label{eq_rel_Z_nuevamb1}
	\e\left[\exp\left\{-\lambda Z_{n-m}^{[\dot{q}_{m+1}-1]\oplus Q_{m+1}}\right\}\right] 
	&=&
	\frac{1}{f'_{m+1}(1)}f'_{m+1}(f_{m+1,n}(e^{-\lambda})),\\
	\label{eq_rel_Z_nuevamb2}
	\e\left[\exp\left\{-\lambda Z_{n-m}^{[\ddot{q}_{m+1}-2]\oplus Q_{m+1}}\right\}\right] 
	&=&
	\frac{1}{\nu_{m+1} f'_{m+1}(1)^2}f''_{m+1}(f_{m+1,n}(e^{-\lambda})), 
	\\
	\label{eq_rel_Zpunto}
	\e\left[\exp\left\{-\lambda\dot{Z}_n^Q\right\}\right] 
	&=& \frac{1}{\mu_n}f'_{0,n}(e^{-\lambda}) e^{-\lambda},
	\\
	\label{eq_rel_Zpun_tras}
	\e\left[\exp\left\{-\lambda \dot{Z}_{n-(m+1)}^{ Q_{m+1}}\right\}\right] 
	&=&
	\frac{\mu_{m+1}}{\mu_n}f'_{m+1,n}(e^{-\lambda })e^{-\lambda },
	\\
	\label{eq_rel_Zdospunt}
	\e\left[\exp\left\{-\lambda \ddot{Z}_{n}^{ Q}\right\}\right] 
	&=& \frac{1}{\mu_n^2}\left(\sum_{k=0}^{n-1}\frac{\nu_{k+1}}{\mu_{k}}\right)^{-1}f''_{0,n}(e^{-\lambda})e^{-2\lambda}.
	\end{eqnarray}
\end{lem}

\begin{proof}
	Denote by \ $(g_{m+1}, f_{m+2}, f_{m+3}, \dots)$ \ the generating functions of the environment \ $ [\dot{q}_{m+1}-1]\oplus Q_{m+1} = ([\dot{q}_{m+1}-1],q_{m+2}, q_{m+3}, \dots),$ \ where $\dot{q}_{m+1}$ is given in \eqref{defqdot}. Note that, 
	\begin{equation*}
	g_{m+1}(s) 
	=  \frac{1}{f'_{m+1}(1)}\sum_{k=1}^{\infty} ks^{k-1} q_{m+1}(k) =\frac{1}{f_{m+1}'(1)} f'_{m+1}(s).
	\end{equation*}
	Then we can deduce \eqref{eq_rel_Z_nuevamb1}, i.e.
	\begin{equation*}
	\e\left[\expo{-\lambda Z_{n-m}^{[\dot{q}_{m+1}-1]\oplus Q_{m+1}}}\right] 
	= g_{m+1} \circ f_{m+2} \circ \cdots\circ f_n(e^{-\lambda }) 
	=  \frac{1}{f'_{m+1}(1)}f'_{m+1}(f_{m+1,n}(e^{-\lambda })),
	\end{equation*}
	where we use the probability generating function of a GWVE given in \eqref{laplace Z}.
	The proof of \eqref{eq_rel_Z_nuevamb2} follows similar arguments. Recall the definition of $\ddot{q}_{m+1}$ in \eqref{defqddot}. It is enough to see that the generating function of \ $[\ddot{q}_{m+1}-2]$, \    denoted by \ $h_{m+1}$, \ is 
	\begin{equation*}
	h_{m+1}(s) 
	=  \frac{1}{\nu_{m+1}f_{m+1}'(1)^2}\sum_{k=2}^{\infty} k(k-1)s^{k-2}q_{m+1}(k)
	=\frac{1}{\nu_{m+1}f_{m+1}'(1)^2} f''_{m+1}(s).
	\end{equation*}
	
	In order to prove \eqref{eq_rel_Zpunto}, note that \ $\dot{Z}_n^Q$ \ is a size-biased transform of \ $Z_n^Q$. \ Then, by \eqref{conditionated}
	$$\int_{0}^{\lambda} \e\left[\exp\left\{-s\dot{Z}_n^Q\right\}\right]\ud s =\frac{\e\left[1-\exp\left\{-\lambda Z_n^Q\right\} \mid   Z_n^Q>0 \right]}{\e\left[Z_n^Q \mid Z_n^Q > 0\right]}=\frac{\e\left[1-\exp\left\{-\lambda Z_n^Q\right\}  \right]}{\e\left[Z_n^Q \right]},$$
	for all \ $\lambda\geq 0$. \ Differentiating the previous equation with respect to \ $\lambda$ \ and using the generating function of \ $Z_n^Q$, \ we obtain
	\begin{equation*}
	\e\left[\exp\left\{-\lambda\dot{Z}_n^Q\right\}\right] = \frac{1}{\mu_n} \frac{\ud }{\ud \lambda}(1-f_{0,n}(e^{-\lambda})) 
	= \frac{1}{\mu_n} f'_{0,n}(e^{-\lambda}) e^{-\lambda}.
	\end{equation*}
	
	The identity \eqref{eq_rel_Zpun_tras} is obtained similarly to  \eqref{eq_rel_Zpunto} but instead of working with the original environment \ $Q$ \ we use the shifted environment \ $Q_{m+1}$. 
	
	Finally, in order to obtain \eqref{eq_rel_Zdospunt} we use the decomposition presented in Proposition \ref{prop_desco_Zdosptos} 
	$$\small{ \e\left[e^{-\lambda \ddot{Z}^Q_n}\right] 
		=\e\left[e^{-\lambda \dot{Z}^Q_n}\right] \sum_{m=0}^{n-1} \p(K_n = m) \e\left[\expo{-\lambda \dot{Z}_{n-(m+1)}^{Q_{m+1}}} \right]
		\frac{\e\left[\expo{-\lambda Z_{n-m}^{[\ddot{q}_{m+1}-2]\oplus Q_{m+1}}}\right]}{\e\left[\expo{-\lambda Z_{n-m}^{[\dot{q}_{m+1}-1]\oplus Q_{m+1}}}\right]}}.$$
	Remember that \ $K_n$ has distribution \eqref{eq_ley_Kn}. Hence, substituting the previous Laplace transforms (i.e. equations \eqref{eq_rel_Z_nuevamb1},\eqref{eq_rel_Z_nuevamb2} and \eqref{eq_rel_Zpunto}) and simplifying, we get
	\begin{equation*}
	\begin{split}
	\e\left[e^{-\lambda \ddot{Z}^Q_n}\right] 
	=& \frac{f'_{0,n}(e^{-\lambda}) }{\mu_n}e^{-\lambda}  \sum_{m=0}^{n-1}  \frac{\nu_{m+1}}{\mu_{m}}
	\left(\sum_{k=0}^{n-1}\frac{\nu_{k+1}}{\mu_{k}}\right)^{-1}\frac{\mu_{m+1}}{\mu_n}  
	\frac{f'_{m+1,n}(e^{-\lambda})e^{-\lambda}}{\nu_{m+1} f'_{m+1}(1)}\frac{ f''_{m+1}(f_{m+1,n}(e^{-\lambda}))}{f'_{m+1}(f_{m+1,n}(e^{-\lambda}))}\\
	=& e^{-2\lambda} \frac{1}{\mu_n^2}\left(\sum_{k=0}^{n-1}\frac{\nu_{k+1}}{\mu_{k}}\right)^{-1} f'_{0,n}(e^{-\lambda}) \sum_{m=0}^{n-1}  
	f'_{m+1,n}(e^{-\lambda})\frac{f''_{m+1}(f_{m+1,n}(e^{-\lambda}))}{ f'_{m+1}(f_{m+1,n}(e^{-\lambda}))}.
	\end{split}
	\end{equation*}
	Note that for all \ $s\in[0,1]$ \ and \ $0\leq m< n$,
	$$f'_{m+1, n}(s) = \prod_{l=m+2}^{n}f'_l(f_{l,n}(s))=\frac{\prod_{l=1}^{n}f'_l(f_{l,n}(s))}{\prod_{l=1}^{m+1}f'_l(f_{l,n}(s))}= \frac{f'_{0,n}(s)}{f'_{m+1}(f_{m+1,n}(s)) \prod_{l=1}^{m}f'_l(f_{l,n}(s))}.$$
	Then, 
	\begin{equation*}
	\begin{split}
	\e\left[e^{-\lambda \ddot{Z}^Q_n}\right] 
	=&e^{-2\lambda} \frac{1}{\mu_n^2}\left(\sum_{k=0}^{n-1}\frac{\nu_{k+1}}{\mu_{k}}\right)^{-1}f'_{0,n}(e^{-\lambda})^2 \sum_{m=0}^{n-1}   \frac{f''_{m+1}(f_{m+1,n}(e^{-\lambda}))}{ f'_{m+1}(f_{m+1,n}(e^{-\lambda}))^2\prod_{l=1}^{m}f'_l(f_{l,n}(e^{-\lambda}))}\\
	=&e^{-2\lambda} \frac{1}{\mu_n^2}\left(\sum_{k=0}^{n-1}\frac{\nu_{k+1}}{\mu_{k}}\right)^{-1}f''_{0,n}(e^{-\lambda}).
	\end{split}
	\end{equation*}
	This completes the proof.
\end{proof}

The next lemma provides the uniform convergence of the function $g$ defined in \eqref{eq_g}. It is essentially saying that whether we start a critical Galton-Watson process with $[\ddot{q}_{\cdot}-2]$ or $[\dot{q}_{\cdot}-1]$, the distribution at large times does not change a lot. The reader will find its importance  in the next Section. In particular, from Proposition \ref{prop_desco_Zdosptos}, the lemma gives the precise meaning of equation \eqref{aprox}.

\begin{lem}\label{lemma g}
	Suppose that Condition \eqref{eq_cond_mild} is fulfilled. Then, for any $\lambda \geq 0$, 
	\begin{equation*}
	\underset{n\rightarrow \infty}{\lim}\sup_{0\leq m <n} \sup_{s \in[0,\lambda]}  \left(1-g\left(n,m,\frac{s}{a_n^Q}\right)\right) =0.
	\end{equation*}
	
\end{lem}
\begin{proof}
	By applying Lemma \ref{prop_relat} , we have that for any \ $s\in[0,\lambda]$ \ and \ $0\leq m\leq n-1$, 	
	\begin{equation*}
	g\left(n,m,\frac{s}{a_n^Q}\right) = \frac{f'_{m+1}(1)}{f'_{m+1}(f_{m+1,n}(e^{-s/a_n^Q}))}\frac{f''_{m+1}(f_{m+1,n}(e^{-s/a_n^Q}))}{f_{m+1}''(1)}.
	\end{equation*}
	The proof is thus complete as soon as we can show the following uniform convergences 
	\begin{eqnarray}
	\underset{n\rightarrow \infty}{\lim}	\sup_{0\leq m < n} \sup_{s \in[0,\lambda]}  \left(1-\frac{f'_{m+1}(f_{m+1,n}(e^{-s/a_n^Q}))}{f'_{m+1}(1)}\right) &=&0,\label{eq_unifcon1}\\
	\underset{n\rightarrow \infty}{\lim}	\sup_{0\leq m  < n} \sup_{s \in[0,\lambda]}  \left(1-\frac{f''_{m+1}(f_{m+1,n}(e^{-s/a_n^Q}))}{f_{m+1}''(1)}\right) &=& 0\label{eq_unifcon2}.
	\end{eqnarray}
	We shall start with \eqref{eq_unifcon1}. With the help of the Mean Value Theorem  for \ $f'_{m+1}$ \ and using that \ $f''_{m+1}$ \ is increasing, we obtain 
	\begin{equation*}
	\begin{split}
	0\leq\sup_{0\leq m  < n} \sup_{s \in[0,\lambda]}  \left(1-\frac{f'_{m+1}(f_{m+1,n}(e^{-s/a_n^Q}))}{f'_{m+1}(1)}\right)
	&\leq\sup_{0\leq m  < n} \sup_{s \in[0,\lambda]} \frac{f_{m+1}''(1)}{f_{m+1}'(1)} \left(1-\
	f_{m+1,n}(e^{-s/a_n^Q})\right).
	\end{split}
	\end{equation*}
	Kersting \cite[Equation 23]{kersting2017unifying} showed that under Condition \eqref{eq_cond_mild}, there exists  \ $ c>0$ \ such that
	\begin{equation}\label{ecuacion 23}
	f_k''(1)\leq cf_k'(1)(1+f_k'(1)), \qquad \mbox{ for all }k\geq 1.
	\end{equation}
	Thus
	\begin{equation*}\small
	\begin{split}
	\sup_{0\leq m < n} \sup_{s \in[0,\lambda]}  \left(1-\frac{f'_{m+1}(f_{m+1,n}(e^{-s/a_n^Q}))}{f'_{m+1}(1)}\right)
	&\leq\sup_{0\leq m < n} \sup_{s \in[0,\lambda]} c (1+f_{m+1}'(1))\left(1-\
	f_{m+1,n}(e^{-s/a_n^Q})\right).
	\end{split}
	\end{equation*}
	
	For similar argument to those given above, using Condition \eqref{eq_cond_mild}, and upon an adjustment of the value of the constant, we can get the same upper bound for the left-hand side supremums in \eqref{eq_unifcon2}. Therefore, it is enough to prove 
	\begin{equation}\label{eq_convunif4}
	\underset{n\rightarrow \infty}{\lim}\sup_{0\leq m < n} \sup_{s \in[0,\lambda]} (1+f_{m+1}'(1)) (1-\
	f_{m+1,n}(e^{-s/a_n^Q})) = 0.
	\end{equation}
	Let $\lambda \geq 0$. By the Mean Value Theorem for \  $f_{m+1,n}$ \ and using that \ $f_{m+1,n}'$ \ is an increasing function, we get for any $0\leq s\leq \lambda $ and $0\leq m< n$
	\begin{equation*}
	0\leq (1+ f_{m+1}'(1))	\left(1-f_{m+1,n}(e^{-s/a_n^Q})\right)  \leq (1+f_{m+1}'(1)) f'_{m+1,n}(1)(1-e^{-s/a_n^Q}).
	\end{equation*}
	Observe that by Taylor's approximation, \ $e^{-s/a_n^Q} = 1- \tfrac{s}{a_n^Q} + y_n$ \ where \ $y_n\geq 0$ is the remainder error term.\ Then,  
	for \ $s\in[0,\lambda]$\  and \ $0\leq m< n$
	\begin{equation}\label{eq_convunf3}
	0\leq(1+ f_{m+1}'(1))\left(1-f_{m+1,n}(e^{-s/a_n^Q})\right)   \leq (1+f_{m+1}'(1))\frac{\mu_n}{\mu_{m+1}} \frac{\lambda}{a_n^Q}=\left(\frac{1}{\mu_{m+1}} + \frac{1}{\mu_m}\right) \frac{\mu_n}{a_n^Q} \lambda.
	\end{equation}
	
	Now, we decompose the left-hand side of \eqref{eq_convunif4} into two limits  where the supremum is taken over two separate sets.  Recall that in the critical case, given an \ $\epsilon>0$ \ there exists \ $N>0$ \ such that \ $\left(a_k^Q\right)^{-1}\leq \epsilon $ \ for any \ $k\geq N$.  Then, we take the two sets as \ $\{m< N\}$ \ and \  $\{N\leq m<n\}$. For the first limit, we observe  
	\begin{equation*}
	\sup_{0\leq m <N}\sup_{s \in[0,\lambda]} (1+f_{m+1}'(1))\left(1-f_{m+1,n}(e^{-\lambda/a_n^Q})\right) \leq \frac{\mu_n}{a_n^Q} \lambda \max_{0\leq  m < N} \left(\frac{1}{\mu_{m+1}} + \frac{1}{\mu_m}\right).
	\end{equation*}
	By criticality, \ $\mu_n/a_n^Q\rightarrow 0$ \ as \ $n\rightarrow 0$. \ Then,   
	\begin{equation}\label{eq_convunif1}
	\lim\limits_{n\to \infty}\sup_{0\leq m <N}\sup_{s \in[0,\lambda]} (1+f_{m+1}'(1))\left(1-f_{m+1,n}(e^{-\lambda/a_n^Q})\right) = 0.
	\end{equation}
	
	For the second limit, we note  that for any $0\leq m\leq n$,
	$$\frac{a_m^Q}{\mu_m}=\frac{1}{2}\underset{k=0}{\overset{m-1}{\sum}}\frac{\nu_{k+1}}{\mu_k}\leq \frac{1}{2}\underset{k=0}{\overset{n-1}{\sum}}\frac{\nu_{k+1}}{\mu_k}=\frac{a_n^Q}{\mu_n}.$$
	Then, by \eqref{eq_convunf3}  and using that \ $N\leq m <n$ \ we get 
	\begin{equation*}\label{eq_convunif2}
	\sup_{N\leq m < n} \sup_{s \in[0,\lambda]} (1+f'_{m+1}(1))\left(1-f_{m+1,n}(e^{-\lambda/a_n^Q})\right)
	\leq\lambda  
	\sup_{N\leq m < n} \left(\frac{1}{a_{m+1}^Q} 
	+ \frac{1}{a_{m}^Q}\right)\leq 2\epsilon \lambda.
	\end{equation*}
	Therefore,
	$$\lim\limits_{n\to \infty}\sup_{N\leq m < n}\sup_{s \in[0,\lambda]} (1+f_{m+1}'(1))\left(1-f_{m+1,n}(e^{-\lambda/a_n^Q})\right) = 0,$$
	which together with the  limit \eqref{eq_convunif1} gives us  \eqref{eq_convunif4}.  This concludes the proof.
\end{proof}

In particular, from Proposition \ref{prop_desco_Zdosptos}, we can deduce that \eqref{aprox} holds.

\section{Proof of the main result}
As we explained in the outline of the proof, in this manuscript we provide a probabilistic argument of a  Yaglom-type limit for critical GWVEs. In the previous section we deduced that \ $\ddot{Z}_n^Q$ \ is the \ $Z_n^Q(Z_n^Q-1)$-transform of \ $Z_n^Q$ and that equation \eqref{aprox} holds. Here, we prove the other remaining steps, contained in Proposition \ref{Prop_unifor} and Proposition \ref{limit Zdot}. First, we present these propositions. Then, using all the tools that we created, we provide a proof for our main result. Finally, we prove the two propositions.

Recall the definition of \ $K_n$ \ in \eqref{eq_ley_Kn}. Given the environment \ $Q$, \ we define
$$	A_{n,m}:= \frac{a^{Q_{m+1}}_{n-(m+1)}}{a_n^Q},\qquad \mbox{ for }  0\leq m<n.$$

\begin{proposition}\label{Prop_unifor} 
	Let $Z^Q$ be a critical GWVE satisfying Condition \eqref{eq_cond_mild}. Then
	\begin{equation*}
	A_{n,K_n}\stackrel{(d)}{\longrightarrow} U, \qquad  \text{as }\ n \rightarrow \infty,
	\end{equation*}
	where \ $U$ \ is an uniform random variable on \ $[0,1]$.  
\end{proposition}

Using the previous proposition, we can show the following.

\begin{proposition}\label{limit Zdot}
	Let \ $\dot{Z}^Q=\{\dot{Z}_n^Q: n\geq 0\}$ \ be a size-biased GWVE. Then, 
	$$(a_n^Q)^{-1}\dot{Z}_n^Q\overset{(d)}{\longrightarrow}\dot{Y} \qquad \mbox{ as }\quad n\rightarrow \infty,$$
	where \ $\dot{Y}$ \ is the size-biased transform of a standard exponential random variable. 
\end{proposition}
We have all the ingredients to prove Yaglom's Theorem under Assumption \eqref{eq_cond_mild}.

\begin{proof}[Proof of Theorem \ref{teo_yaglom}]
	According to Lemma \ref{remark 1},  in order to deduce Theorem \ref{teo_yaglom}, it  is enough to  show  that  \ $(a_n^Q)^{-1}\dot{Z}_n^Q\overset{(d)}{\longrightarrow}\dot{Y}$ \ and  \ $\e\left[(a_n^Q)^{-1}Z_n^Q\mid Z_n^Q>0\right]\longrightarrow 1$\ as \ $n\rightarrow \infty$,\  where \ $\dot{Y}$ \ is the size-biased transform of an exponential random variable. The first limit holds by Proposition \ref{limit Zdot}. For the second limit, we  observe that 
	$$\e\left[\left.\frac{Z_n^Q}{a_n^Q}\right| Z_n^Q>0\right]=\frac{\e\left[Z_n^Q\right]}{a_n^Q\p\left( Z_n^Q>0\right)}=\frac{\mu_n}{a_n^Q\p\left( Z_n^Q>0\right)},$$
	which goes to 1 according to \eqref{limit a}. Therefore, Theorem \ref{teo_yaglom} holds.
\end{proof}

This manuscript is complete as soon as we prove Propositions \ref{Prop_unifor} and \ref{limit Zdot}. We start with Proposition \ref{Prop_unifor}.

\begin{proof}[Proof of Proposition \ref{Prop_unifor}]
	In order to obtain this result, it is enough to deduce
	\begin{equation}
	\label{distF_n}
	\lim\limits_{n\to \infty} \p\left( A_{n,K_n}\leq y\right) = y, \qquad y\in[0,1].
	\end{equation} 
	Denote by \ $(\tilde{f}_1, \tilde{f}_2,\dots)$ \ the generating functions associated with the environment \ $Q_{m+1}$. \ They can be written in terms of the original environment as  \ $ \tilde{f}_k = f_{m+1+k}$, \ for \ $k\geq 1$. \
	Then, by definition
	$$\tilde{\mu}_k=f'_{m+2}(1) \cdots f_{m+1+k}'(1)
	= \frac{\mu_{m+1+k}}{\mu_{m+1}}\qquad \mbox{ and } \qquad 
	\tilde{\nu}_k =\frac{f''_{m+1+k}(1)}{f'_{m+1+k}(1)^2}= \nu_{m+1+k}.$$
	Hence,
	$$a^{Q_{m+1}}_{n-(m+1)} 
	= \frac{\tilde{\mu}_{n-(m+1)}}{2}\sum_{k=0}^{n-(m+1)-1} \frac{\tilde{\nu}_{k+1}}{\tilde{\mu}_{k}} 
	= \frac{\mu_{n}}{2}\sum_{j=m+1}^{n-1} \frac{\nu_{j+1}}{\mu_{j}},$$
	and
	\begin{equation*}
	\begin{split}
	A_{n,m}= & \frac{a^{Q_{m+1}}_{n-(m+1)} }{a^Q_n}=\sum_{j=m+1}^{n-1} \frac{\nu_{j+1}}{\mu _{j}} \left(\sum_{k=0}^{n-1} \frac{\nu_{k+1}}{\mu _{k}}\right)^{-1}
	= 1 - \sum_{j=0}^{m} \frac{\nu_{j+1}}{\mu _{j}} \left(\sum_{k=0}^{n-1} \frac{\nu_{k+1}}{\mu _{k}}\right)^{-1},
	\end{split}
	\end{equation*}
	where in the last equality, we completed the sum. Then, 
	\begin{equation}\label{eq_proba_Kn}
	\begin{split}
	\p\left(A_{n,K_n} = 1 - \sum_{j=0}^{m} \frac{\nu_{j+1}}{\mu _{j}} \left(\sum_{k=0}^{n-1} \frac{\nu_{k+1}}{\mu _{k}}\right)^{-1}\right)
	&=\p\left(A_{n,K_n}=A_{n,m}\right)=\p\left(K_n=m\right)\\
	&
	=\frac{\nu_{m+1}}{\mu_{m}}\left(\sum_{k=0}^{n-1}\frac{\nu_{k+1}}{\mu_{k}}\right)^{-1}.
	\end{split}
	\end{equation}
	
	Note that  \ $\{A_{n,m}: m= 0, \dots, n-1\}\subset[0, 1]$ \ is a decreasing sequence with \ $A_{n,n-1}=0$. \ Then, we can associate it to the partition \ $P^{(n)}= \{0= \Pi^{(n)}_0 < \Pi^{(n)}_1 < \ldots < \Pi_{n-1}^{(n)}<\Pi^{(n)}_n=1\}$ \  defined by \  $\Pi^{(n)}_k = A_{n,n-k-1}$, \ for any \ $0\leq k<n$, \ with \ $\Pi^{(n)}_n=1$.\ The norm of the partition is defined by 
	$$ || P^{(n)}|| = \max_{1\leq k \leq n} \left\{\Pi^{(n)}_k - \Pi^{(n)}_{k-1}\right\}=  \max_{0\leq  m \leq n-1} \left\{\frac{\nu_{m+1}}{\mu _{m}} \left(\sum_{k=0}^{n-1} \frac{\nu_{k+1}}{\mu _{k}}\right)^{-1}\right\}.$$
	
	Since \ $P^{(n)}$ \ is a partition, for each \ $y\in [0,1)$ \ there exists \ $l_n:= l(y,n)\in \{0,1,\ldots, n-1\}$ \ such that \ $\Pi_{l_n}^{(n)} \leq y  <\Pi_{l_n+1}^{(n)}.$ \ Then, by \eqref{eq_proba_Kn} 
	\begin{equation*}
	\p\left(A_{n,K_n}\leq y\right) =\sum_{k=0}^{l_n}\p\left(A_{n,K_n}=\Pi^{(n)}_k\right)=\sum_{m=n-l_n-1}^{n-1} \frac{\nu_{m+1}}{\mu_{m}}\left(\sum_{k=0}^{n-1}\frac{\nu_{k+1}}
	{\mu_{k}}\right)^{-1}=\Pi_{l_n+1}^{(n)}.
	\end{equation*}
	It is easy to deduce that in order to prove \eqref{distF_n}, we have to prove that \ $\Pi_{l_n+1}^{(n)}\rightarrow y$ \ as $n\rightarrow \infty$. \ We always choose \ $l_n$ \ such that \ $y\in[\Pi_{l_n}^{(n)},\Pi_{l_n+1}^{(n)})$. \ Therefore, it is enough to show that \ $|| P^{(n)}|| \rightarrow 0$\  as \ $n \rightarrow \infty$.
	
	From inequality \eqref{ecuacion 23}, we see that for each \ $1\leq n$,
	\begin{equation*}
	\begin{split}
	\frac{\nu_n}{\mu_{n-1}} \left(\sum_{k=0}^{n-1}\frac{\nu_{k+1}}{\mu_{k}}\right)^{-1} &= \frac{f_n''(1)}{f_{n}'(1)}	\left( \mu_n\sum_{k=0}^{n-1}\frac{\nu_{k+1}}{\mu_{k}}\right)^{-1} \leq c(1+f_n'(1)) \left(\mu_n\sum_{k=0}^{n-1}\frac{\nu_{k+1}}{\mu_{k}}\right)^{-1} \\ 
	&=   c \left(\mu_n\sum_{k=0}^{n-1}\frac{\nu_{k+1}}{\mu_{k}}\right)^{-1} + c \left(\mu_{n-1}\sum_{k=0}^{n-2}\frac{\nu_{k+1}}{\mu_{k}} + \nu_n\right)^{-1}.
	\end{split}
	\end{equation*}
	Since we are in the critical regime and \ $\nu_n\geq 0$ \ for all \ $n\geq 1$,\ both summands in the right-hand side of the last equality go to zero as \ $n\rightarrow \infty$. \ In other words, given \ $\epsilon>0$ \ there exists \ $N\geq 1$ \ such that
	\begin{equation}\label{eq_lim1}
	\frac{\nu_{m+1}}{\mu_{m}} \left(\sum_{k=0}^{n-1}\frac{\nu_{k+1}}{\mu_{k}}\right)^{-1} \leq\frac{\nu_{m+1}}{\mu_{m}} \left(\sum_{k=0}^{m}\frac{\nu_{k+1}}{\mu_{k}}\right)^{-1} \leq \epsilon, \qquad \qquad  \mbox{for any }\  N\leq m< n.
	\end{equation}
	On the other hand, by criticality, for any fixed \ $m\leq N$,  \ there is a $M_m\in \n$ such that 
	\begin{equation}\label{eq_lim2}
	\frac{\nu_{m+1}}{\mu_{m}} \left(\sum_{k=0}^{n-1}\frac{\nu_{k+1}}{\mu_{k}}\right)^{-1} \leq \epsilon, \qquad \qquad  \mbox{for any }\  n\geq M_m.
	\end{equation}
	We
	define \ $M= N\vee \max\{M_m: m\leq N\}$. \ Then, by \eqref{eq_lim1} and \eqref{eq_lim2}, for any \ $n\geq M$
	\begin{equation*}
	|| P^{(n)}|| =   \max_{0\leq  m <n} \left\{\frac{\nu_{m+1}}{\mu _{m}} \left(\sum_{k=0}^{n-1} \frac{\nu_{k+1}}{\mu _{k}}\right)^{-1}\right\} \leq  \epsilon,
	\end{equation*}
	and the claim is true.
\end{proof}

Now, we present a result whose relevance will become clear  in the proof of Proposition \ref{limit Zdot}. Intuitively, the first statement is an extension of the fact that \ $A_{n,K_n}$ \ converges in distribution to \ $U$. \ The limit of \ $B_3^{(n)}$ \ is an extension of the fact that \ $g\left(n,m,s/a_n^Q\right)$ \ converges uniformly to \ 1. \ For the purpose of seeing the intuition in the statement of \ $B_2^{(n)}$, \ we normalise \ $\dot{Z}_{n-(m+1)}^{Q_{m+1}}$ \ with the correct constant corresponding to the shifted environment. Then,  \ $B_2^{(n)}$ \ infers that at large times  the distributions of the processes  \ $\dot{Z}_n^Q/a_n^Q$ \ and \ $\dot{Z}_{n-(m+1)}^{Q_{m+1}}/a_{n-(m+1)}^{Q_{m+1}}$ \ do not vary much.

\begin{lem}\label{Bs}
	Let $Q$ be a varying environment satisfying Condition \eqref{eq_cond_mild} and \ $\{\dot{Z}^{\cdot}_n: n\geq 0\}$ \ be a  size-biased GWVE.  Define
	\begin{equation*}
	\begin{split}
	B_1^{(n)} &
	= \int_{0}^{\lambda} \left( \e\left[\expo{-sU \frac{\dot{Z}_n^Q}{a_n^Q}}\right]-\sum_{m=0}^{n-1} \p(K_n=m) \e\left[\expo{-sA_{n,m}\frac{\dot{Z}_n^Q}{a_n^Q}}\right]  \right)\ud s,\\
	B_2^{(n)} &
	= \int_{0}^{\lambda} \sum_{m=0}^{n-1} \p(K_n=m)  \left( \e\left[\expo{-sA_{n,m}\frac{\dot{Z}_n^Q}{a_n^Q}}\right]- \e\left[\expo{-s\frac{\dot{Z}_{n-(m+1)}^{Q_{m+1}}}{a_n^Q}}\right]  \right)\ud s,
	\end{split}
	\end{equation*}
	
	\begin{equation*}
	\begin{split}
	B_3^{(n)} &
	=\int_{0}^{\lambda} \sum_{m=0}^{n-1} \p(K_n=m)  \e\left[\expo{-s\frac{\dot{Z}_{n-(m+1)}^{Q_{m+1}}}{a_n^Q}}\right] \left(1-g\left(n,m,\frac{s}{a_n^Q}\right)\right)\ud s,
	\end{split}
	\end{equation*}
	where \ $U$ \ is an uniform random variable on \ $[0,1]$ \ independent of \ $\dot{Z}^Q$.	\ Then,
	\begin{equation*}
	\limsup_{n \to \infty} |B_1^{(n)}|  
	= \limsup_{n \to \infty} |B_2^{(n)}| = \limsup_{n \to \infty} |B_3^{(n)}| = 0.
	\end{equation*}
\end{lem}

\begin{proof}
	We start with \ $B_1^{(n)}$. Recall the partition \  $P^{(n)}= \{ \Pi^{(n)}_0 < \Pi^{(n)}_1 < \ldots < \Pi_{n-1}^{(n)}<\Pi_n^{(n)}\}$ \ given in the proof of Proposition \ref{Prop_unifor} and that \ $\p(K_n=m)=\Pi^{(n)}_{n-m} - \Pi^{(n)}_{n-m-1}$. \ Then 
	\begin{equation*}
	\begin{split}
	b_1^{(n)}(s):=
	& \e\left[\expo{-sU \frac{\dot{Z}_n^Q}{a_n^Q}}\right]-\sum_{m=0}^{n-1} \p(K_n=m) \e\left[\expo{-sA_{n,m}\frac{\dot{Z}_n^Q}{a_n^Q}}\right]  \\
	=& \int_{0}^{1}\e\left[\expo{-su \frac{\dot{Z}_n^Q}{a_n^Q}}\right] \ud u-\sum_{m=0}^{n-1} (\Pi^{(n)}_{n-m} - \Pi^{(n)}_{n-m-1}) \e\left[\expo{-sA_{n,m}\frac{\dot{Z}_n^Q}{a_n^Q}}\right]. 
	\end{split}
	\end{equation*}
	By decomposing \ $[0,1]$ \ into the subintervals \ $[\Pi^{(n)}_{n-m-1},\Pi^{(n)}_{n-m}],\ m=0,\ldots,n-1$,\ we get
	\begin{equation*}
	\begin{split}
	b_1^{(n)}(s)=& \sum_{m=0}^{n-1} \int_{\Pi^{(n)}_{n-m-1}}^{\Pi^{(n)}_{n-m}}\e\left[\expo{-su \frac{\dot{Z}_n^Q}{a_n^Q}}\right] \ud u- \sum_{m=0}^{n-1}\int_{\Pi^{(n)}_{n-m-1}}^{\Pi^{(n)}_{n-m}} \e\left[\expo{-sA_{n,m}\frac{\dot{Z}_n^Q}{a_n^Q}}\right]\ud u.
	\end{split}
	\end{equation*}
	Now, by Lemma \ref{prop_relat}, the Laplace transform of \ $\dot{Z}_n^Q$ \ can be expressed in terms of  \ $f_{0,n}'$. Since \ $x\mapsto f_{0,n}'(e^{-\lambda x})e^{-\lambda x}$ \ is a decreasing function, and \  $u, A_{n,m}\in[ \Pi_{n-m-1}^{(n)},\Pi^{(n)}_{n-m} ]$ \  for \ $m=0,\ldots, n-1$, we deduce
	\begin{equation*}
	\begin{split}
	|	b_1^{(n)}(s) |\leq
	&  \frac{1}{\mu_n}\sum_{m=0}^{n-1} \int_{\Pi^{(n)}_{n-m-1}}^{\Pi^{(n)}_{n-m}}
	\left| f_{0,n}'\left(e^{-su/a_n^Q}\right)e^{-su/a_n^Q}- f_{0,n}'\left(e^{-sA_{n,m}/a_n^Q}\right)e^{-sA_{n,m}/a_n^Q}  \right|\ud u\\ 
	\leq & \frac{1}{\mu_n} \sum_{m=0}^{n-1} (\Pi^{(n)}_{n-m}-\Pi^{(n)}_{n-m-1})\\
	&\hspace{1.5cm}\times 
	\left(f'_{0,n}\Bigg(e^{-s\Pi^{(n)}_{n-m-1}/a_n^Q}\Bigg)
	e^{-s\Pi^{(n)}_{n-m-1}/a_n^Q}
	-f'_{0,n}\Bigg(e^{-s\Pi^{(n)}_{n-m}/a_n^Q}\Bigg)
	e^{-s\Pi^{(n)}_{n-m}/a_n^Q} \right).
	\end{split}
	\end{equation*}
	The last sum can be bounded by the norm of the partition multiplied by a telescopic sum with \ $\Pi_0^{(n)}=0$ \ and \ $  \Pi_n^{(n)}=1$. \ Therefore
	$$|b_1^{(n)}(s)|\leq \frac{1}{\mu_n} || P^{(n)}||\left(f'_{0,n}(1) - f'_{0,n}\left(e^{-s/a_n^Q}\right) e^{-s/a_n^Q}\right)  \leq  \frac{1}{\mu_n} f'_{0,n}(1) || P^{(n)}||  = || P^{(n)}||.$$
	Since the norm of the partition goes to zero as $n\rightarrow \infty$ (see the proof of Proposition \ref{Prop_unifor}), we get the result for \ $B_1^{(n)}$,
	$$		\limsup_{n \to \infty} |B_1^{(n)}|\leq\limsup_{n \to \infty} \int_0^{\lambda}|b_1^{(n)}(s)|\ud s\leq 	\limsup_{n \to \infty}\lambda || P^{(n)}||=0.$$ 
	
	Now we deal with \ $B_2^{(n)}$.\ By Lemma \ref{prop_relat}, the Laplace transform of \ $\dot{Z}_n^Q$ \ and $\dot{Z}_{n-m-1}^{Q_{m+1}}$ \ can be expressed in terms of  \ $f_{0,n}'$ \ and \ $f_{m+1,n}'$, \ respectively. Then,
	\begin{equation*}
	\begin{split}
	b_{2}^{(n,m)}(s):=& \e\left[\expo{-sA_{n,m}\frac{\dot{Z}_n^Q}{a_n^Q}}\right]- \e\left[\expo{-s\frac{\dot{Z}_{n-(m+1)}^{Q_{m+1}}}{a_n^Q}}\right]\\
	= &\frac{1}{\mu_n}
	f_{0,n}'\left(e^{-sA_{n,m}/a_n^Q}\right)e^{-sA_{n,m}/a_n^Q}- \frac{\mu_{m+1}}{\mu_n}f_{m+1,n}'\left(e^{-s/a_{n}^Q}\right)e^{-s/a_{n}^Q}. 
	\end{split}   
	\end{equation*}
	Using first the Fundamental Theorem of Calculus and then the Mean Value Theorem in the functions \ $f_{0,n}$ \ and \ $f_{m+1,n}$, \ we deduce that  
	\begin{equation}\label{eq B2}
	\begin{split}
	&\int_0^{\lambda}b_{2}^{(n,m)}(s)\ud s\\
	& =\frac{1}{\mu_n}\left(\frac{a_n^Q}{A_{n,m}}\left(f_{0,n}(1)-f_{0,n}\left(e^{-\lambda A_{n,m}/a_n^Q}\right) \right) -
	\mu_{m+1}a_{n}^Q \left(f_{m+1,n}(1)-f_{m+1,n}\left( e^{-\lambda/a_n^Q}\right)\right)\right)
	\\  
	&=\frac{1}{\mu_n}\left(\frac{a_n^Q}{A_{n,m}}\left(1-e^{-\lambda A_{n,m}/a_n^Q}\right) f'_{0,n}(\xi)-
	\mu_{m+1}a_n^Q \left(1-e^{-\lambda/a_n^Q}\right) f'_{m+1,n}(\eta)\right),
	\end{split}
	\end{equation}
	where \ $\xi \in \left(e^{-\lambda A_{n,m}/a_n^Q},1\right)$ \ and \ $\eta \in \left(e^{-\lambda/a_n^Q},1\right)$. \ Now, we shall find \ $\widehat{B}_2^{(n)}$ \ and \ $\widetilde{B}_2^{(n)}$\ 
	such that  \ $\widehat{B}_2^{(n)}\rightarrow 0 $, \ $\widetilde{B}_2^{(n)}\rightarrow 0 $ \ as \ $n\rightarrow \infty$ \ and 
	\begin{equation}
	\label{intervalo}
	B_2^{(n)}   = \sum_{m=0}^{n-1}\p(K_n=m) \int_0^{\lambda}b_{2}^{(n,m)}(s)\ud s\in \left[\widehat{B}_2^{(n)}, \widetilde{B}_2^{(n)}\right].
	\end{equation}
	The fact that \ $f'_{0,m}$ \ and \ $f'_{m,n} $ \ are increasing functions and \eqref{eq B2} imply the following lower bound  
	\begin{equation*}\begin{split}
	B_2^{(n)}\geq  \widehat{B}_2^{(n)}:=&
	\sum_{m=0}^{n-1} \frac{\p(K_n=m)}{\mu_n} \left(\frac{a_n^Q}{A_{n,m}}\left(1-e^{-\lambda A_{n,m}/a_n^Q}\right) f'_{0,n}(e^{-\lambda A_{n,m}/a_n^Q})\right.\\
	&\hspace{3cm}
	\left.-\mu_{m+1}a_n^Q \left(1-e^{-\lambda/a_n^Q}\right) f'_{m+1,n}(1)\right).
	\end{split}
	\end{equation*}
	Since \ $\p(K_n=m)=\p(A_{n,K_n}=A_{n,m})$ \ and \ $\mu_{m+1}f'_{m+1,n}(1)=\mu_n$,\ we have
	\begin{equation*}\begin{split}\label{b2hat}
	\widehat{B}_2^{(n)}= \e\left[ \frac{a_n^Q}{A_{n,K_n}}\left(1-e^{-\lambda A_{n,K_n}/a_n^Q}\right)\frac{ f'_{0,n}\left(e^{-\lambda A_{n,K_n}/a_n^Q}\right)}{f'_{0,n}(1)} \right]-a_n^Q\left(1-e^{-\lambda/a_n^Q}\right).
	\end{split}
	\end{equation*}

	On the other hand, for the upper bound, observe that \ $\mu_{m+1}=f'_1(1)\cdots f'_{m+1}(1)$,\ that  \ $f_{l,n}(e^{-\lambda/a_n^Q}) \leq 1$ \ for any \ $1\leq l \leq m+1$, \ and  that \ $f_l'$ \ is an increasing function, then 
	$$
	\mu_{m+1}f'_{m+1,n}(e^{-\lambda/a_n^Q})\geq \underset{l=1}{\overset{m+1}{\prod}}f_l'\left(f_{l,n}\left(e^{-\lambda /a_n^Q}\right)\right)f'_{m+1,n}\left(e^{-\lambda /a_n^Q}\right)= f'_{0,n}\left(e^{-\lambda /a_n^Q}\right),$$
	where in the equality we use \eqref{formula f''}. By \eqref{eq B2} and \eqref{intervalo}, the previous inequality implies that 
	\begin{equation*}\begin{split}\label{b2tilde}
	B_2^{(n)}\leq 	\widetilde{B}_2^{(n)}:=&
	\sum_{m=0}^{n-1} \frac{\p(K_n=m)}{\mu_n} \left(\frac{a_n^Q}{A_{n,m}}\left(1-e^{-\lambda A_{n,m}/a_n^Q}\right) f'_{0,n}(1)\right.\\
	&\hspace{3cm}
	\left.-a_n^Q \left(1-e^{-\lambda/a_n^Q}\right) f'_{0,n}(e^{-\lambda/a_n^Q})\right)
	\\
	=&  \e\left[ \frac{a_n^Q}{A_{n,K_n}}\left(1-e^{-\lambda A_{n,K_n}/a_n^Q}\right) \right]-a_n^Q\left(1-e^{-\lambda/a_n^Q}\right) \frac{ f'_{0,n}\left(e^{-\lambda /a_n^Q}\right)}{f'_{0,n}(1)}.
	\end{split}
	\end{equation*}
	Then, \eqref{intervalo} holds. Now, we show that the limit of \  $\widehat{B}_2^{(n)}$ \ and \ $\widetilde{B}_2^{(n)}$ \ is zero as \ $n\rightarrow \infty$. \ Recall that \ $0\leq A_{n,K_n}\leq 1$ \ and \ $a_n^Q\rightarrow \infty$\ as\ $n\rightarrow \infty$, \ then as \ $n\rightarrow \infty$ 
	$$ a_n^Q\left(1-e^{-\frac{\lambda}{a_n^Q}}\right) \rightarrow \lambda, \qquad	 \qquad
	\frac{ f'_{0,n}\left(e^{-\lambda/a_n^Q}\right)}{f'_{0,n}(1)} \rightarrow 1,  $$
	and
	$$ \frac{a_n^Q}{A_{n,K_n}}\left(1-e^{-\lambda A_{n,K_n}/a_n^Q}\right) \rightarrow \lambda, \qquad \mbox{and} \qquad \frac{ f'_{0,n}\left(e^{-\lambda A_{n,K_n}/a_n^Q}\right)}{f'_{0,n}(1)} \rightarrow 1 \qquad \mbox{a.s.}$$
	By Dominated Convergence Theorem,  we have that \ $\widehat{B}_2^{(n)}\rightarrow 0 $ \ and \ $\widetilde{B}_2^{(n)}\rightarrow 0 $ \ as \ $n\rightarrow \infty$. \ Therefore, \ $B_2^{(n)}$ \ has the same behaviour.  Since the limit is zero we also have that  $|B_2^{(n)}| \rightarrow 0$ \ as  \ $n\to \infty$.
	
	Finally, we deal with \ $B_3^{(n)}$. \ Given an \ $\epsilon >0$, \ by Lemma \ref{lemma g}, there exists \ $M>0$ \ such that for \ $n\geq M,\ 0\leq s\leq \lambda$ \ and \ $0\leq m <n$
	$$ \left| g\left(n,m,\frac{s}{a_n^Q}\right) -1\right| \leq \epsilon.$$
	Hence, for $n\geq M$,
	$$|B_3^{(n)}| \leq  \int_{0}^{\lambda} \sum_{m=0}^{n-1} \p(K_n=m)   \left| g\left(n,m,\frac{s}{a_n^Q}\right)-1\right|\ud s \leq \epsilon \lambda.$$
	Since \ $\epsilon$ \ is arbitrary, we get that \ $\limsup_{n\to \infty}|B_3^{(n)}| = 0$. 
\end{proof}

For the proof of Proposition \ref{limit Zdot}, we need the following two lemmas, which the reader can find  in \cite[Lemma 3.1 and Lemma 3.2]{ren20172}. The first lemma compares the generating functions of two variables with the generating functions of their size-biased transforms. The second lemma is similar to Gr\"onwall's Lemma.

\begin{lem}\label{lemmaRen3.1}
	Let \ $X$ \ and \ $W$ \ be two non-negative random variables with mean \ $\mu$. \ Let \ $F$ \ and \ $G$ \ be functions such that \ $\e\left[e^{-\lambda \dot{X}}\right]=\e\left[e^{-\lambda X}\right]F(\lambda)$ \ and \ $\e\left[e^{-\lambda \dot{W}}\right]=\e[e^{-\lambda W}]G(\lambda)$, \ where \ $\dot{X}$ \ and \ $\dot{W}$ \ are the size-biased transforms of \ $X$ \ and \ $W$. \ Then, 
	$$\left|\e\left[e^{-\lambda X}\right]-\e\left[e^{-\lambda W}\right]\right|\leq\mu \left|\int_0^{\lambda} \left(F(s)-G(s)\right)\ud s\right|, \qquad \lambda\geq 0.$$
\end{lem}

\begin{lem}\label{lemmaRen3.2}
	Suppose that a non-negative bounded function \ $F$ \ on \ $[0,\infty)$ \ and a constant \ $c>0$ \ satisfy
	$$F(\lambda)\leq c \int_0^1\ud u\int_0^{\lambda}F(us)\ud s, \qquad \mbox{ for } \lambda\geq 0.$$
	Then \ $F\equiv 0$.
\end{lem}

Finally, we present the last proof in this manuscript.
\begin{proof}[Proof of Proposition \ref{limit Zdot}]
	We define the bounded function 
	$$ M(\lambda) = \limsup_{n \to \infty}\left| \e[e^{-\lambda \dot{Y}}] -\e\left[\expo{-\lambda \frac{\dot{Z}_n^Q}{a_n^Q}}\right] \right|, \qquad \mbox{ for } \lambda\geq 0.$$
	We will use Lemma \ref{lemmaRen3.1} with \ $X=\dot{Y}$ \ and \ $W=(\dot{Z}_n^Q-1)/a_n^Q$. \ Since \ $Y$ \ is an exponential random variable and \eqref{eq_media_Zn} holds, we get
	\begin{equation*}
	\e[\dot{Y}]=2= \e\left[\frac{\dot{Z}_n^Q -1}{a_n^Q}\right].
	\end{equation*} 
	Thanks to characterisation \eqref{eq_expone}, we may choose \ $F(\lambda)=\e[e^{-\lambda U\dot{Y}}]$, \ where\  $U$\  is an uniform variable on \ $[0,1]$ \ independent of \  $\dot{Y}$. \ Then, by Proposition \ref{prop_desco_Zdosptos}, we have
	$$G(\lambda)=\sum_{m=0}^{n-1} \p(K_n = m) g\left(n,m, \frac{\lambda}{a^Q_n}\right)\e\left[\expo{-\lambda \frac{\dot{Z}_{n-(m+1)}^{Q_{m+1}}}{a_{n}^{Q}} } \right],$$
	where \ $g $ \ is given in \eqref{eq_g}. Hence, by Lemma \ref{lemmaRen3.1}  and the triangle inequality,  
	\begin{equation*}
	\begin{split}
	&\hspace{1cm}\Bigg| \e\left[e^{-\lambda \dot{Y}}\right]-\e\left[\expo{-\lambda \frac{\dot{Z}^Q_n-1}{a^Q_n}}\right]  \Bigg|\\
	&\hspace{1cm}\leq  2 \left| \int_{0}^{\lambda} \left(\e[e^{-s U\dot{Y}}]- \sum_{m=0}^{n-1} \p(K_n = m) g\left(n,m, \frac{s}{a^Q_n}\right)\e\left[\expo{-s \frac{\dot{Z}_{n-(m+1)}^{Q_{m+1}}}{a_{n}^{Q}} } \right]\right)
	\ud s\right|  \label{eq_descoZY}\\
	&\hspace{1cm} \leq2\left(|B_1^{(n)}|+|B_2^{(n)}|+|B_3^{(n)}|+|B_4^{(n)}|\right),
	\end{split}
	\end{equation*}
	where \ $  B_1^{(n)} , \ B_2^{(n)} $ \ and \ $B_3^{(n)} $\ are defined in Lemma \ref{Bs} and 
	$$B_4^{(n)} 
	= \int_{0}^{\lambda}  \left(\e\left[e^{-sU \dot{Y}}\right] - \e\left[\expo{-sU\frac{\dot{Z}_n^Q}{a_n^Q}}\right] \right)\ud s,$$
	with \ $U$ \ a uniform random variable on \ $[0,1]$ \ independent of \ $\dot{Y}$ \ and \ $\dot{Z}^Q$. \ Then, by Lemma \ref{Bs} and the Dominated Convergence Theorem, we obtain
	$$M(\lambda)\leq 2\limsup_{n \to \infty }\int_{0}^{\lambda} \left| 
	\e\left[e^{-sU \dot{Y}}\right]- \e\left[\expo{-sU\frac{\dot{Z}_n^Q}{a_n^Q}}\right]  \right|\ud s \leq 2\int_{0}^{\lambda}\int_{0}^{1} M(us)\ud u \ud s.$$
	By Lemma \ref{lemmaRen3.2}, \  $M\equiv 0$ \ which  implies that \ ${\dot{Z}_n^Q}/{a_n^Q}$ \ converges weakly to \ $\dot{Y}$. 
\end{proof}

\section*{Acknowledgements}
N.C.-T. acknowledges support from CONACyT-MEXICO grant no. 636133.  S.P. is a Newton International Fellow Alumnus (AL191032). S.P would like to thank Yan-Xia Ren and Zhenyao Sun for a discussion of their paper. We are grateful to Juan Carlos Pardo for his careful reading of an earlier version of this manuscript. This research is supported by UNAM-DGAPA-PAPIIT grant no. IA103220. The authors thank the referees for their careful reading of our paper and for their constructive comments.

\bibliographystyle{abbrv}
\bibliography{referencias}

\end{document}